\documentclass[12pt,a4paper,leqno]{amsart}
\usepackage{preamble}
\usepackage{microtype}
\usepackage{adjustbox}
\usepackage{paperspecific}

\usepackage{thmtools}

\title{Multivariate Frequent Stability and Diam-Mean Equicontinuity}
\author{Lino Haupt}
\date{\today}

\begin{document}
\begin{abstract}
	In this paper, we introduce and investigate multivariate versions of frequent stability and diam-mean equicontinuity.
	Given a natural number $m > 1$, we call those notions \enquote{frequent $m$-stability} and \enquote{diam-mean $m$-equicontinuity}.
	We use these dynamical rigidity properties to characterise systems whose factor map to the maximal equicontinuous factor (MEF) is finite-to-one for a
	residual set, called \enquote{almost finite-to-one extensions}, or a set of full measure, called \enquote{almost surely finite-to-one extensions}.

	In the case of a $\sigma$-compact, locally compact, abelian acting group
	it is shown that frequently $(m+1)$-stable systems are equivalently characterised as almost $m$-to-one extensions of their MEF.
	Similarly, it is shown that a system is diam-mean $(m+1)$-equicontinuous if and only if it is an almost surely $m$-to-one extension of its MEF.
\end{abstract}

\maketitle
\tableofcontents
\newpage
\section*{Introduction}
Dynamical systems are broadly categorised into two classes: chaotic and ordered.
The transition between chaos and order is a rich area of study,
explored through various perspectives, contexts, and definitions.
The paradigmatic examples of ordered systems are the \textit{equicontinuous} ones.
They are precisely those exhibiting \textit{continuous} dependence on initial conditions.
According to the classical dichotomy established by \textsc{Auslander} and \textsc{Yorke} \cite{originalauslanderyorkedichotomy},
a minimal system is either equicontinuous or exhibits \textit{sensitive} dependence on initial conditions,
a hallmark of chaos.
However, not all systems with this sensitive dependence are considered truly chaotic.
For the study of those systems, the \textit{maximal equicontinuous factor} (MEF) plays a critical role as the best equicontinuous approximation.
Understanding the factor map to the MEF offers insights into dynamical rigidity and the nature of their potentially chaotic behavior.

In order to better understand the transition between chaos and order various \textit{weak} versions of equicontinuity have been proposed.
Such equicontinuity properties are \enquote{mean equicontinuity}, introduced in \cite{Fomin1951} under the name \enquote{mean $L$-stability}, and \enquote{frequent stability} and \enquote{diam-mean equicontinuity}, both introduced in \cite{diammeanequicts}.
In an \enquote{equicontinuous} system \textit{all} iterates of close starting point to remain close together.
Moreover, in a \enquote{mean equicontinuous} system the iterates of close starting points only are close \textit{on average}.
Furthermore, in a \enquote{diam-mean equicontinuous} system the \textit{diameter} of the iterates of a small ball is small \textit{on average}.
Finally, in a \enquote{frequently stable} system the diameter of the iterates of a small ball shall be small for a \textit{positive density} of times.

Common to all of those notions is that they study the distance between the iterates of \textit{two} starting points or the diameter,
as the maximal distance that any \textit{pair} of points exhibit, of the iterates of a small ball.
Furthermore, they are characterised as extensions of the MEF which are \textit{one-to-one} in a suitable sense.
Recently, there has been an effort to generalise those notions to a \textit{multivariate} setting,
\ie by studying multivariate distances between \textit{multiple} starting points or the multivariate diameter,
as the maximal multivariate distance that any \textit{tuple} of points exhibit, of the iterates of a small ball.
Furthermore, the goal is to characterise those as extensions of the MEF which are \textit{finite-to-one} in a suitable sense.
\enquote{Multivariate equicontinuity} and related notions have recently been studied in \cite{FelipeMultivariateEquicontinuity}.
Among other things the authors prove that a system is $m$-equicontinuous if and only if the fibers to the MEF contain at most $m-1$
points.
As they additionally prove an Auslander-Yorke type dichotomy they recover \cite[Theorem 3.6]{Shao2008}.
\enquote{Multivariate mean equicontinuity} has been studied in \cite{breitenbücher2024multivariatemeanequicontinuityfinitetoone}.
In this paper we introduce and study both \enquote{multivariate frequent stability} and \enquote{multivariate diam-mean equicontinuity} and characterise them by suitable finite-to-one extension properties.

In order to be more precise let us fix some notation.
Formally, a topological dynamical system (tds) is an action $\alpha$ of a locally compact group acting $G$ on a compact metric space $X$, written as a triple $(X,\alpha,G)$.
In order to keep this introduction simple we will here (unlike the rest of the paper) only consider $\Z$ actions.
We write them as a pair $(X,\varphi)$,
where $X$ is a compact metric space and $\varphi: X \to X$ is a homeomorphism.

In the past years, there has been interest in linking and characterising the above mentioned weak equicontinuity properties
through the invertibility properties of the factor map to the MEF.

Formally, a system $(X,\varphi)$ is called \emph{mean equicontinuous} if 
for any $\varepsilon > 0$ there is $\delta > 0$ such that
$d(x,y) < \delta$ implies that 
\[ \limsup_{N \to \infty} \frac{1}{N} \sum_{n=0}^{N-1} d\kl \varphi^n(x), \varphi^n(y) \kr < \varepsilon \,.\]
Those systems are exactly those whose factor map to the MEF is measure-theoretically an isomorphism,
as proven in \cite{isoext} by \textsc{Downarowicz} and \textsc{Glasner} (for the non-minimal case see \cite{topoiso}).\footnote{Those works built on \cite{LI_TU_YE_2015}.}

Moreover, a system $(X,\varphi)$ is called \emph{frequently stable} if
for any $\varepsilon>0$ there is a $\delta > 0$ such that
\[ \ml n \in \N \mm \diam\kl \varphi^n(B_\delta(x)) \kr < \varepsilon \mr \]
has positive density for any $x \in X$, where $\diam(A)$ is its usual diameter.
Recently, \textsc{Xu} and \textsc{Hu} \cite{frequentlystableisalmostautomorphic} 
proved that a minimal system is frequently stable if and only if the factor map to the MEF is almost 1:1.

Finally, $(X,\varphi)$ is called \emph{diam-mean equicontinuous} if
for any $\varepsilon>0$ there is a $\delta > 0$ such that
\[ \limsup_{N \to \infty} \frac{1}{N} \sum_{n=0}^{N-1} \diam\kl \varphi^n(B_\delta(x)) \kr < \varepsilon \]
for all $x \in X$.
Again in \cite{diammeanequicts}, the authors show that a minimal system is diam-mean equicontinuous
if and only if its factor map to the MEF is injective on a set of full measure (almost surely 1:1)s

All of those results are in the regime of one-to-one extensions.
However, there are many systems whose factor map to the MEF is finite-to-one (in various meanings of the word).
The perhaps most prominent example of this kind is the Thue-Morse subshift (e.\,g.~\cite{Kakutani1972StrictlyES}).
Its fibers to the MEF contain at most 4 elements and almost surely only 2.
Indeed, by \cite[Section III, Theorem 4]{Dekking1978},
any constant length substitution yields a subshift that is an almost surely finite-to-one extension of its MEF.
Similar structures can be observed in the class of model sets (e.\,g.~\cite{irregularmodelsets}).

We here propose multivariate analogues of the notion of frequent stability and diam-mean equicontinuity.
We follow our work in \cite{breitenbücher2024multivariatemeanequicontinuityfinitetoone},
in which a multivariate notion of mean equicontinuity is introduced and it is proven that this property is satisfied by systems whose factor map to the MEF is measure-theoretically finite-to-one.
A key concept is the family of multivariate analogues of metrics given by
\[ \Dm: X^m \longrightarrow \R^+_0, \; \Dm(x_1,\ldots,x_m) := \min_{1\leqslant i < j \leqslant m} d(x_i,x_j) \,,\]
where $m \in \N$.
Using this notion, one can define a multivariate diameter
\[ \diam_m: \Pot X \to \R^+_0, \; \diam_m(A) := \sup \ml  \Dm(a_1,\ldots,a_m)  \mm (a_1,\ldots,a_m) \in A^m \mr \,. \]
We call $(X,\varphi)$ \emph{frequently $m$-stable} if
for any $\varepsilon > 0$ there is $\delta > 0$ such that
\[ \ml n \in \N \mm \diam_{m}\kl \varphi^n(B_\delta(x)) \kr < \varepsilon \mr \]
has positive density for any $x \in X$.
Furthermore, we call $(X,\varphi)$ \emph{diam-mean $m$-equicontinuous} if
for any $\varepsilon > 0$ there is $\delta > 0$ such that
\[ \limsup_{N \to \infty} \frac{1}{N} \sum_{n=0}^{N-1} \diam_m\kl \varphi^n(B_\delta(x)) \kr < \varepsilon \]
for all $x \in X$.
Now let $\pi: (X,\varphi) \to (Y,\psi)$ be the factor map to the MEF.
Define \enquote{almost $m$:1} and \enquote{almost surely $m$:1} as the existence a residual, respectively a full measure set of points which have at most $m$ preimages.
Then our main results can be stated as follows.

\begin{theoremA}[\refthm{mainthmfreqstable}]
	Let $(X,\varphi)$ be minimal.
	Then $(X,\varphi)$ is frequently $(m+1)$-stable if and only if $\pi$ is almost $m$:1.
\end{theoremA}

\begin{theoremB}[\refthm{mainthmdiamequi}]
	Let $(X,\varphi)$ be minimal.
	Then $(X,\varphi)$ is diam-mean $(m+1)$-equicontinuous if and only if $\pi$ is almost surely $m$:1.
\end{theoremB}

All this will be done in the context of $\sigma$-compact, locally-compact, abelian group actions on compact metric spaces (instead of only considering $\Z$-actions).

The structure of this paper is as follows:
Section \ref{sec:prelim} covers all general preliminary results used in the proofs and derives properties of the multivariate diameter.
The relevant preliminary results regarding the MEF are presented in Section \ref{sec:MEF}.
The finite-to-one invertibility properties, \emph{almost finite-to-one} and \emph{almost surely finite-to-one},
are defined and explored in Section \ref{sec:invertibility}.
Finally, the main results, are discussed and proved in Sections \ref{sec:frequentlystable} and \ref{sec:diammean}.

{\textsc{Acknowledgements.}}
I would like to express my gratitude to \textsc{Xiangdong Ye}, who helped me in locating sources.
He proved, together with \textsc{Wen Huang} and \textsc{Ping Lu}, a result similar to \refthm{multivariateregionallyproximal} for $\Z$ actions using different means in \cite{yemultivariateregionallyproximal}.
Furthermore, he pointed me to the article \cite{auslandermultivariateregionallyproximal} in a private communication.

Further, I want to thank my advisor \textsc{Tobias Jäger} for his helpful remarks and the suggestion of the topic
and my colleagues \textsc{Jamal Drewlo} and \textsc{Zeyu Kang} for fruitful discussions.

\section{Preliminaries}
\label{sec:prelim}
\emph{
	Throughout this paper, we assume that
	$G$ is a $\sigma$-compact, locally compact, abelian group.
	Furthermore, $X$ is a compact metric space and $\alpha: G \times X \to X$
	a continuous group action.
}

Given a set $A$, its power set is denoted by $\Pot A$ and its cardinality is written as $\card(A)$. 
If $B$ is another set, then $B^A$ is the set of functions from $A$ to $B$.
The set of continuous functions between two metric spaces $A$ and $B$ is denoted by $\Cc(A,B)$.
Furthermore, $\delta_z$ denotes the Dirac measure in $z$.
Finally we set $\N := \ml 1, 2, 3, \ldots  \mr$.

Let $(X,\alpha,G)$ be a tds.
The \textdef{orbit} of $x \in X$, given by $\ml \alpha(g,x) \mm g \in G \mr$, is denoted by $\alpha(G,x)$.
If all orbits are dense, we call $(X,\alpha,G)$ \textdef{minimal}.
Let $(Y,\beta,G)$ be another tds.
A map $\pi: X \to Y$ is called a \textdef{factor map}
if $\pi$ is surjective and $\pi(\alpha(g,x)) = \beta(g,\pi(x))$ for any $g \in G$ and $x \in X$.
In this case, we write $\pi: (X,\alpha,G) \to (Y,\beta,G)$.
If $\pi$ is a homeomorphism, it is called a \textdef{conjugacy}.

For the convenience of the reader, some classical and general preliminary results
will be listed before specific preliminaries are discussed in detail.
The results stated here are already adapted to our standing assumptions for the sake of simplicity.

Recall the following definition
\begin{defn}
	A family $\Fc \subseteq \Cc(X,Y)$ is called \textdef{equicontinuous}
	if for any $x \in X$ and $\varepsilon > 0$ there is $\delta > 0$ such that
	for any $f \in \Fc$ we have 
	$f(B_\delta(x)) \subseteq B_\varepsilon(f(x))$.
\end{defn}

\begin{proventhm}[Arzelà-Ascoli, {\cite[Theorem 7.17, p.~233 et seq.]{kelley}}]
	\label{thm:ascoli}
	Let $\Cc(X,Y)$ be equipped with the topology of uniform convergence.
	Then $\Fc \subseteq \Cc(X,Y)$ is compact if and only if $\Fc$ is closed and equicontinuous.
\end{proventhm}
\begin{proventhm}[Dini, {\cite[p.~239]{kelley}}]
	\label{thm:dini}
	Let $(f_n)_{n\in\N} \in \Cc(X,\R)^\N$. 
	If $(f_n)_{n\in\N}$ converges monotonically and pointwise
	to a continuous function $f \in \Cc(X,\R)$,
	then the convergence is uniform.
\end{proventhm}
\begin{proventhm}[Baire Category, {\cite[p.~200]{kelley}}]
	\label{thm:baire}
	If $A \subseteq X$ is residual, i.\,e.~contains a countable intersection of open and dense sets, then $A$ is dense.
\end{proventhm}
The following result, although trivial in our setting, tends to make arguments more lucid. 
\begin{provenlem}[Urysohn, {\cite[p.~115]{kelley}}]
	\label{lem:urysohn}
	For any pair of closed and disjoint subsets $A,B \subseteq X$
	there is a continuous function $f:X \to [0,1]$
	such that $f(A) = \ml 0 \mr$ and $f(B) = \ml 1 \mr$.
\end{provenlem}

\subsection{Averaging and Banach density}
The existence and uniqueness of invariant measures on locally compact groups is established by a celebrated theorem due to \textsc{Haar}.
In the version proven by \textsc{Weil} it states that
\begin{proventhm}[{\cite{weilsproofofhaar}}]
	\label{thm:haar}
	On any locally compact group $\Gamma$,
	there is a (unique up to scaling) locally-finite Borel measure $m_\Gamma$,
	such that $m_\Gamma(\gamma B) = m_\Gamma(B)$ for any Borel set $B\subseteq \Gamma$ and any $\gamma \in \Gamma$. 
\end{proventhm}
For the rest of this work, we pick a Haar measure $m$ on $G$.

For non-compact groups, these measures are infinite.
As with $\Z$ and $\R$, we will thus take averages using subsets of finite size.
This is done using
\begin{defn}[F\o lner Sequence]
	\label{def:folnersequence}
	A sequence $\Fc=(F_n)_{n\in\N}$ is called \textdef{F\o lner}
	if
	$F_n \subseteq G$ is compact and
	for any compact set $K \subseteq G$ and $\varepsilon > 0$ there is $N \in \N$ such that
	$\infty > m(F_n) > 0$ and
	\begin{align}
		\label{eq:folnernet}
		\frac{m\kl F_n \triangle K F_n \kr}{m\kl F_n \kr} < \varepsilon \,.
	\end{align}
	for all $n > N$.
\end{defn}

In our setting, the existence of such sequences is guaranteed by the following result.
\begin{proventhm}[e.\,g.~\cite{Pier1984AmenableLC}]
	Any abelian, locally compact and $\sigma$-compact group admits at least one F\o lner sequence. \qedhere
\end{proventhm}

We can now quantify the \enquote{average size} of certain subsets.
Let $\Fc$ be a F\o lner sequence and $H \subseteq G$.
\begin{defn}[$\Fc$-density]
	The \textdef{upper $\Fc$-density} of $H$ is defined as
	\[ \Du_\Fc(H) := \limsup_{n \to\infty} \frac{m( H \cap F_n)}{m(F_n)} \,. \]
	Similarly, the \textdef{lower $\Fc$-density} of $H$ is defined as
	\[ \Dl_\Fc(H) := \liminf_{n\to\infty} \frac{m( H \cap F_n)}{m(F_n)} \,.\]
\end{defn}

In some constructions, the set whose density we we aim to capture is shifted at each step.
We therefore define
\begin{defn}[$\Fc$-Banach Density]
	Define the \textdef{upper $\Fc$-Banach density} of $H$ as
	\[ \BDu_\Fc(H) := \limsup_{n\to\infty} \sup_{g\in G} \frac{m( H \cap gF_n)}{m(gF_n)}\]
	for any $H \subseteq G$.
	Similarly, the \textdef{lower $\Fc$-Banach density} of $H$ is defined as
	\[ \BDl_\Fc(H) := \liminf_{n\to\infty} \inf_{g\in G} \frac{m( H \cap gF_n)}{m(gF_n)} \,.\]
\end{defn}
A straightforward argument shows that
\begin{provenlem}
	It holds for any $H \subseteq G$ that
	\[ \BDl_\Fc(H) \leqslant \Dl_\Fc(H) \leqslant \Du_\Fc(H) \leqslant \BDu_\Fc(H) \]
	as well as
	\[ \Du_\Fc(H) = 1 - \Dl_\Fc(H^c) \quad\text{and}\quad \BDu_\Fc(H) = 1 - \BDl_\Fc(H^c) \,. \qedhere \]
\end{provenlem}

\subsection{Ergodicity}
Let $\mu$ be a measure on $X$ and $f:X \to Z$.
By $f_*\mu$ we denote the \textdef{pushforward measure} with $f_*\mu(B) = \mu(f\inv(B))$.
A probability measure $\mu$ is called \textdef{($\alpha$-)invariant}
if $\alpha(g,\cdot)_*\mu = \mu$ for any $g \in G$.
Note that every invariant measure is assumed to be a probability measure, by definition,
in order to simplify notation.
A subset $A \subseteq X$ is called \textdef{($\alpha$-)invariant}
if $\alpha(g\inv,A) \subseteq A$ for any $g \in G$.
An invariant measure $\mu$ is called \textdef{ergodic} if $\mu(A) \in \ml 0, 1 \mr$ for every invariant set $A$.
A tds is called uniquely ergodic, if there is exactly one invariant measure.
In this case, the unique invariant measure is always ergodic.

The Riesz Representation Theorem \cite[Theorem 7.44]{functional}
and the Banach-Alaoglu Theorem \cite[8.10 and 8.11]{functional}
jointly imply
\begin{proventhm}[{\cite[Proposition 8.27]{functional}}]
	\label{thm:alaoglu}
	The space of probability measures on any compact metric space equipped with the weak-$*$-topology is a compact metric space itself.
\end{proventhm}

This prvides the basis for the Krylov-Bogolyubov Procedure, which allows to construct invariant measures by averaging:
\begin{proventhm}[{\cite[Theorem 8.10]{Einsiedler2011}}]
	\label{thm:krylovbogolyubov}
	For a sequence $(\mu_n)_{n\in \N}$ of probability measures and any F\o lner sequence $\Fc = (F_n)_{n\in\N}$
	define $\nu_n$ by the formula
	\[ \int f \dd\nu_n := \frac{1}{m(F_n)}\int_{F_n} \int f \circ \alpha(g,\cdot) \dd\mu_n \dd m(g)\,, \]
	where $f \in \Cc(X,\R)$.
	Then any weak-$*$-cluster point $\nu$ of $(\nu_n)_{n\in \N}$ is $\alpha$-invariant.
\end{proventhm}

One of the consequences is the Uniform Ergodic Theorem.
\begin{proventhm}[{\cite[Theorem 4.10]{Einsiedler2011}}]
	\label{thm:uniformergodicthm}
	If $(Y,\beta,G)$ is uniquely ergodic with invariant measure $\nu$,
	then
	\[ \lim_{n\to\infty} \frac{1}{m(F_n)} \int_{F_n} f \circ \beta(g,x) \dd m(g) = \int f \dd\nu \,. \]
	uniformly in $x \in X$ for any F\o lner sequence $\Fc=(F_n)_{n\in\N}$ and any $f\in \Cc(X,\R)$.
\end{proventhm}

\begin{cor}
	\label{cor:uniformergodicandbanach}
	Let $(Y,\beta,G)$ be uniquely ergodic with invariant measure $\nu$.
	If $f \geqslant \1_A$ is continuous,
	then
	\[
		\BDl_\Fc\kl \ml g \in G \mm f(\beta(g,y)) > 0 \mr \kr \geqslant \nu(A)\,.
	\] 
	for any F\o lner sequence $\Fc$ and $y \in Y$.
	Analogously, if $f \leqslant \1_A$ is continuous, then
	\[
		\BDu_\Fc\kl \ml g \in G \mm f(\beta(g,y)) > 0 \mr \kr \leqslant \nu(A)\,.
	\] 
	for any F\o lner sequence $\Fc$ and $y \in Y$.
\end{cor}
\begin{proof}
	Let $\varepsilon > 0$.
	By \refthm{uniformergodicthm}, the ergodic averages
	of $f$ converge uniformly
	to $\int f \dd\nu \geqslant \nu(A)$.
	In particular, there is $N \in \N$ such that 
	\begin{align*}
		\frac{1}{m(F_n)} \int_{F_n} f \circ \beta(g,y) \dd m(g)  \geqslant \nu(A) - \varepsilon
	\end{align*}
	for all $n > N$ and any $y \in Y$.
	So for any $n > N$, $y\in Y$ and $h \in G$
	it holds that
	\[ \frac{1}{m(hF_n)} \int_{hF_n} f \circ \beta(g,y) \dd m(g) = \frac{1}{m(F_n)} \int_{F_n} f \circ \beta(g,\beta(h,y)) \dd m(g)  \geqslant \nu(A) - \varepsilon \,. \qedhere\]
\end{proof}

In the non-uniquely ergodic case, there are non-ergodic measures.
However
\begin{proventhm}[Ergodic Decomposition, {\cite[p.~266, Theorem 8.20]{Einsiedler2011}}]
	\label{thm:ergodicdecomp}
	Let $Z$ be a compact metric space and let $\gamma$ be an continuous action of $G$ on $Z$, so that $(Z,\gamma,G)$ is a tds. 
	For any invariant measure $M$ on $Z$, there is a measure $E$ on the set of ergodic measures of $(Z,\gamma,G)$ such that
	\[ \int_Z f \dd M = \int \int_Z f \dd M_e \dd E(M_e) \]
	for any $f \in L^1(\mu)$.
	In short we will write $M = \int M_e \dd E(M_e)$ in that situation.
\end{proventhm}

A F\o lner sequence	$(F_n)_{n\in\N}$ is \textdef{tempered} if and only if
there is $C>0$ such that
\[ \frac{m\kl \bigcup_{k=1}^{n-1} F_k \cdot F_n \kr}{m\kl F_n \kr} < C \]
for all $n \in \N$.
This conditions makes sure that the elements of the F\o lner sequence \enquote{do not move around in $G$ too much}.

\begin{provenlem}[{\cite[Proposition 1.4]{lindenstrauss}}]
	Every F\o lner sequence has a tempered subsequence.
\end{provenlem}

The celebrated Lindenstrauss Ergodic Theorem generalises the one by Birkhoff and states that
\begin{proventhm}[{\cite[Theorem 1.2]{lindenstrauss}}]
	\label{thm:lindenstrauss}
	If $\Fc$ is tempered and $\mu$ is an ergodic measure for $(X,\alpha,G)$, then
	for any $f \in L^1(\mu)$ we have
	\[ \frac{1}{m(F_n)} \int_{F_n} f(\alpha(g,x)) \dd m(g) \konv n \int f \dd\mu \]
	for $\mu$-almost all $x \in X$.
\end{proventhm}

\subsection{Multivariate Diameter}
We extend frequent stability and diam-mean equicontinuity to multivariate settings by introducing a replacement for the diameter.
This generalisation allows us to capture the interactions between multiple points (instead of just two).
This in turn will require the multivariate analogues of the metric introduced in \cite{breitenbücher2024multivariatemeanequicontinuityfinitetoone} to be defined.
Let $m \in \N \setminus \ml 1 \mr$.
We define
\[ \Dm: X^m \longrightarrow  \R_0^+, \; (x_1,\ldots, x_m) := \min_{i \neq j} d(x_i,x_j) \,, \]
measuring the minimal pairwise distance.

By construction, $\Dm$ is symmetric, \ie invariant under any permutation of the entries, and $D_m(x,\ldots,x) = 0$ for any $x \in X$.
Furthermore, it satisfies a generalisation of the triangle inequality as demonstrated in the following lemma.
\begin{provenlem}[{\cite[Lemma 4.2]{breitenbücher2024multivariatemeanequicontinuityfinitetoone}}]
	\label{lem:polygoninequality}
	For any $\mathbf{x}\in X^{m}$ and $z \in X$ we have
	\[ \Dm(\mathbf{x}) \leqslant \sum_{i=1}^m \Dm\Rep{\mathbf{x}}{i}{z} \,, \]
	where $\Rep{\mathbf{x}}{i}{z}$ is the tuple $\mathbf{x}$
	with the $i$-th entry replaced by $z$, i.e.
	\[ \Rep{\mathbf{x}}{i}{z}_j := \begin{cases} z & j = i \\ x_j & j \neq i \end{cases} \,. \qedhere\]
\end{provenlem}
We will call the above inequality the \textdef{polygon inequality}\footnote
{
	The name \enquote{polygon inequality} stems from the fact that we build all possible polygons where the new point $z$ replaces one of the old points $x_i$.
	This might be illustrated by the following (unrelated) more geometric example:
	On $\R^n$ a  multivariate distance for a $n$-tuple of points may be defined as the measure of the convex hull of the points in the tuple.
	The polygon inequality of this multivariate distance results from the fact that the polygons created by letting the new point replace one of the old ones form a disjoint cover of the polygon consisting of the old points.
}.

Furthermore, $D_m$ satisfies a strong continuity requirement.
\begin{lem}
	\label{lem:estimatefordm}
	Let $(x_1,\ldots,x_m), (\tilde x_1,\ldots,\tilde x_m) \in X^{m}$
	If $d(x_i,\tilde x_i) < \varepsilon$ for any $i \in \ml 1,\ldots,m\mr$, then
	\[ \bl D_m(x_1,\ldots,x_m)- D_m(\tilde x_1,\ldots,\tilde x_m) \br <  2\varepsilon \,. \]
\end{lem}
\begin{proof}
	Let $i \neq j$ be such that $D_m(x_1,\ldots,x_m) = d(x_i,x_j)$.
	Now 
	\begin{align*}
		D_m(x_1,\ldots,x_m)  d(x_i, x_j) &\leqslant d(x_i, \tilde x_i) + d(\tilde x_i,\tilde x_j) + d(\tilde x_j,x_j) \\
		&= d(\tilde x_i,\tilde x_j) + 2\varepsilon \\
		&\leqslant D_m(\tilde x_1, \ldots, \tilde x_m) + 2\varepsilon \,.
	\end{align*}
	The other inequality follows analogously.
\end{proof}

For a subset $A \subseteq X$, one usually defines
\[ \diam(A) := \sup \ml  d(a_1,a_2) \mm (a_1,a_2) \in A^2 \mr \,. \]
Using $\Dm$, a multivariate diameter is defined as
\begin{align*}
	\diam_m(A) := \sup\ml  \Dm(a_1,\ldots,a_m) \mm (a_1,\ldots,a_m) \in A^m \mr \,.
\end{align*}
Note that $\diam_2 = \diam$.
\begin{lem}
	For any subset $A \subseteq X$ it holds $\diam_m(A) \in [E_m(A),2\cdot E_m(A)]$, where 
	\[ E_m(A) := \inf\ml \varepsilon > 0 \mm \exists (x_1, \ldots,x_{m-1}) \in X^{m-1} : \bigcup_{i=1}^{m-1} B_\varepsilon(x_i) \supseteq A \mr \,. \]
\end{lem}
\begin{proof}
	Let $\varepsilon > 0$ such that there is $(x_1,\ldots,x_{m-1}) \in X^{m-1}$ with $\bigcup_{i=1}^{m-1} B_\varepsilon(x_i) \supseteq A$.
	Let $(a_1,\ldots,a_m) \in A^m$ be arbitrary.
	By the pigeon-hole principle, there is $k \in \ml 1,\ldots,m-1 \mr$ and $i \neq j$ such that $a_i,a_j \in B_\varepsilon(x_k)$.
	Thus, 
	\[ \Dm(a_1,\ldots,a_m) \leqslant d(a_i,a_j) \leqslant 2\varepsilon \,. \]
	Therefore, $\diam_m(A) \leqslant 2\varepsilon$ and taking the infimum over all such $\varepsilon > 0$ yields
	\[ \diam_m(A) \leqslant2 E_m(A) \,.\]

	For the other inequality, let $\varepsilon > 0$ and assume that
	$A \setminus \bigcup_{i=1}^{m-1} B_\varepsilon(x_i) \neq \emptyset$ for any $(x_1,\ldots,x_{m-1}) \in X^{m-1}$.
	Take any $a_1 \in A$.
	By assumption, there is $a_2 \in A \setminus B_\varepsilon(a_1)$.
	Inductively pick $a_{j+1} \in A \setminus \bigcup_{i=1}^{j} B_\varepsilon(a_i)$ for $j\leqslant m-1$.
	This gives us a tuple $(a_1,\ldots,a_m) \in A^m$ with $\Dm(a_1,\ldots,a_m) > \varepsilon$. 
\end{proof}
The bounds given above are tight as the following example shows.
\begin{ex}
	Let $m \in \N \setminus \ml 1 \mr$ and $X = A = [1,m]$. 
	Observe that $\diam_m(A) = 1$.
	Furthermore, $\bigcup_{i=1}^{m-1} B_{\frac{1}{2}+\delta} (x_i) \supseteq A$ for any $\delta > 0$,
	where $x_i = i+\frac{1}{2}$.
	So in this case $E_m(A) = \frac{1}{2}$ and thus $\diam_m(A) = 2E_m(A)$.

	Conversely, consider $X' = A' = \N \cap [1,m]$.
	Again $\diam_m(A) =1$.
	However, for no $\varepsilon \leqslant 1$ there is $x_1,\ldots,x_{m-1} \in X'$ with $\bigcup_{i=1}^m B_\varepsilon(x_i) \supseteq A'$.
	As clearly $\bigcup_{i=1}^{m-1} B_{1+\delta}(i) \subseteq A'$ for any $\delta > 0$ we have $E_m(A') = 1$.
	So $\diam_m(A') = E_m(A')$. 
\end{ex}

\begin{lem}
	\label{lem:densenessanddiam}
	Let $\varepsilon > 0$ and $m \in \N$.
	Assume that $A,B \subseteq X$ with $B_\varepsilon(A) \supseteq B$.
	Then $\diam_{m}(A) \geqslant \diam_{m}(B) - 2\varepsilon$.
\end{lem}
\begin{proof}
	Pick $b_1,\ldots,b_m \in B$.
	Now there are $a_1,\ldots,a_m \in A$ such that $d(a_i,b_i) < \varepsilon$ for any $i \in \ml 1,\ldots, m\mr$.
	So by \reflem{estimatefordm} we have
	\[ D_m(a_1,\ldots,a_m) \geqslant D_m(b_1,\ldots,b_m) -2\varepsilon \,. \]
	Taking the supremum first over $(a_1,\ldots,a_m) \in A^m$ and then over $(b_1,\ldots,b_m) \in B^m$ we obtain the desired result.
\end{proof}

\subsection{Hausdorff Metric}
The notions of frequent stability and diam-mean equicontinuity consider iterates of $\delta$-balls.
Even though those notions are properties of the original system,
they do not correspond to ergodic averages in that system.
This is because the entire ball, not just its center, is iterated.
In order to still use the ergodic theorem,
it is useful to consider a system in which these averages are indeed ergodic averages.
We do this by lifting the action to the space of compact non-empty subsets and equip this with
the so-called Hausdorff metric, which turns this into a compact metric space.

Equip the set of all compact subsets of $X$, \ie
\[ \Kc(X) := \ml K \subseteq X \mm K \text{ compact }\,, \mr\]
with the \textdef{Hausdorff metric}
\begin{align*}
	d_\Hc(K,W):=& \inf\ml \varepsilon > 0 \mm B_\varepsilon(W) \subseteq K \text{ and } B_\varepsilon(K) \subseteq W \mr \\
	=& \max\ml \sup_{k \in K} \inf_{w \in W} d(k,w), \sup_{w \in W} \inf_{k \in K} d(w,k) \mr \\
	=& \max\ml \max_{k \in K} d(k,W), \max_{w \in W} d(w,K) \mr \,.
\end{align*}
\begin{provenlem}[{\cite[Theorem 4.2]{TopologiesOnSpacesOfSubsets}}]
	\label{lem:convergenceinhausdorffmetric}
	$(\Kc(X),d_\Hc)$ is a compact metric space.
	Furthermore, if $\lim_{n\to\infty} E_n = E$, then
	\[ E = \bigcap_{N=1}^\infty \overline{\bigcup_{n=N}^\infty E_n }\,. \qedhere \]
\end{provenlem}
\begin{provenprop}[{\cite[Remark 4.6.~(3)]{hyperspacecontinuity}}]
	If $(X,\alpha,G)$ is a tds,
	then the action of $\alpha$ on $X$ induces an action $\hat\alpha$ on $(\Kc(X),d_\Hc)$ by
	\[ \hat\alpha(g,K) = \ml \alpha(g,k) \mm k \in K \mr \,. \qedhere \]
\end{provenprop}
Under this action, the ergodic theorems can be applied to the closure of $\delta$-balls.

\reflem{densenessanddiam} implies
\begin{provenlem}
	\label{lem:continuityofdiam}
	$\diam_m: (\Kc,d_\Hc) \to \R^+_0$ is continuous for any $m \in \N$.
\end{provenlem}

In order to understand the topology on $\Kc(X)$, the following is useful.
\begin{lem}
	\label{lem:topologyonkcx}
	If $W \subseteq X$ is compact and $\delta> 0$, then
	\[ C(W,\delta) := \ml K \subseteq X \mm \exists x \in W: \overline{B_\delta(x)} \subseteq K \mr \]
	is a compact subset of $\Kc(X)$.
\end{lem}
\begin{proof}
	Let $(E_n)_{n\in\N} \in C(W,\delta)^\N$ be a sequence.
	Let $(x_n)_{n\in\N} \in W^\N$ be such that $\overline{B_\delta(x)} \subseteq E_n$.
	As $W$ is compact, we can assume that $x_n$ is convergent (else we take a converging subsequence) to $x \in W$.
	The whole space $\Kc(X)$ is compact, so we can furthermore assume that $E_n$ is converging to some $E \in \Kc(X)$.
	We must show that $E \in C(W,\delta)$.
	\reflem{convergenceinhausdorffmetric} implies that
	\[ E = \bigcap_{N=1}^\infty \overline{\bigcup_{n=N}^\infty E_n }\,. \]
	Now, $B_{\delta-\eta}(x) \subseteq \bigcup_{n=N}^\infty E_n$ for any $\eta > 0$ and $N \in \N$, as $\lim_{n\to\infty} x_n = x$.
	This implies that $B_\delta(x) \subseteq \bigcup_{n=N}^\infty E_n$ and thus 
	\[ \overline{B_\delta(x)} \subseteq \bigcap_{N=1}^\infty \overline{\bigcup_{n=N}^\infty E_n } = E\,. \qedhere\]
\end{proof}

In order to study the invariant measures of $\hat\alpha$, we restate \cite[Lemma 3.2]{frequentlystableisalmostautomorphic}.
\begin{lem}
	\label{lem:stabilityofcovering}
	Let $(E_n)_{n\in\N} \in \Kc(X)$ converge to $E \in \Kc(X)$.
	If $F \subseteq G$ is finite with
	$X = \bigcup_{f \in F} \alpha(f,E_n)$ for any $n \in \N$,
	then $X = \bigcup_{f \in F} \alpha(f,E)$.
\end{lem}
Note that \textsc{Xu} and \textsc{Hu} work with semigroups.
For the convenience of the reader we provide a simplified proof for the invertible case.
\begin{proof}
	The action $\alpha(g,\cdot)$ is a homeomorphism for any $g \in G$, so the map
	\[ \alpha(F, \cdot): (\Kc(X),d_\Hc) \longrightarrow (\Kc(X),d_\Hc), \; K \longmapsto  \bigcup_{f \in F} \alpha(f,K ) \]
	is continuous.
	Thus, $X = \alpha(F,E_n) \xlongrightarrow{n \to \infty} \alpha(F,E)$.
\end{proof}

Let us also cite \cite[Lemma 3.3]{frequentlystableisalmostautomorphic}.
Again we provide a proof streamlined for the case of invertible transformations.
\begin{lem}
	\label{lem:coveringimpliesnonemptyinterior}
	Let $E \in \Kc(X)$.
	If there is a countable set $H$ such that $X = \bigcup_{g \in H} \alpha(g\inv,E)$,
	then $E$ has non-empty interior.
\end{lem}
\begin{proof}
	Clearly, $\alpha(g\inv,E)$ is closed, so if it has empty interior it is nowhere dense.
	By the Baire Category \refthm{baire}, at least one of the $\alpha(g\inv,E)$ can not be nowhere dense.
	As $\alpha(g\inv,\cdot)$ is an homeomorphism, $E$ itself has non-empty interior.
\end{proof}

An insightful argument from \cite{frequentlystableisalmostautomorphic} will be used twice in the following.
We formulate it as a proposition.
\begin{prop}
	\label{prop:supportofgeneratedmeasure}
	Let $(X,\alpha,G)$ be minimal and let $B \in \Kc(X)$ be a set with non-empty interior.
	Then any weak-$*$-cluster point $M$ of
	\[ \kl  \frac{1}{m(F_n)} \int_{F_n} \delta_{\hat\alpha(g,B)} \dd m(g) \kr_{n \in \N} \]
	is supported on $J := \ml K \in \Kc(X) \mm \tint(K) \neq \emptyset \mr$.
\end{prop}
\begin{proof}
	By minimality, there is a finite $F \subseteq G$ such that $X = \bigcup_{f \in F} \alpha(f, B)$.
	Now this shows that
	\[ \bigcup_{f \in F} \alpha(f, \alpha(g,B)) = \bigcup_{f \in F} \alpha(fg,B) = \alpha\kl g, \bigcup_{f \in F} \alpha(f,B) \kr = \alpha(g,X) = X \,. \]
	By \reflem{stabilityofcovering}, any point $K$ in the orbit closure of $B$ 
	must satisfy $\bigcup_{f \in F} \alpha(f,K) = X$.
	By \reflem{coveringimpliesnonemptyinterior}, we have $K \in J$.
	Clearly, $M$ is supported on the orbit closure of $B$ and thus on $J$.
\end{proof}
\begin{provencor}
	\label{cor:existenceofergodiccomponentsupportedonJ}
	Let $M$ be supported on $J$ and $E$ be a measure on the ergodic measures such that $M = \int M_e \dd E(M_e)$,
	as in \refthm{ergodicdecomp}.
	As $\int M_e(J) \dd E(M_e) = M(J) = 1$, we have $M_e(J) = 1$ for $E$-almost all $M_e$.
\end{provencor}

\begin{lem}[\cite{frequentlystableisalmostautomorphic}]
	\label{lem:decompositionofJ}
	If a measure $M$ is supported on $J$, then
	there is $B \in J$ such that $M( \ml K \in J \mm B \subseteq K \mr) > 0$.
\end{lem}
\begin{proof}
	Let $D \subseteq X$ be countable and dense.
	Recall the notation from \reflem{topologyonkcx} and observe that
	\[ J = \bigcup_{n \in \N} \bigcup_{d \in D} C\kl \ml d \mr, \frac 1n \kr \,.\]
	So the claim follows by $\sigma$-additivity.
\end{proof}

\subsection{Multivalued Functions}
In order to study the preimages of points,
it is useful to understand \textdef{multivalued functions}.
Those are maps $\Psi: Y \to \Pot X$, written as $\Psi: Y \twoheadrightarrow X$.
The \textdef{upper preimage} $\Psi^u(A)$ of $A \subseteq X$ is given by
\begin{align*}
\Psi^u(A) :=& \ml \high{11} y \in Y \mm \Psi(y) \subseteq A  \mr
\end{align*}

\begin{defn}[Upper Hemi-Continuity]
	\label{def:hemicont}
	$\Psi$ is called
	\textdef{upper} \textdef{hemi-continuous}\footnote
		{
			We follow \cite{infinitedimensional} and use the name \enquote{hemi-continuity}.
			The usage of the name \enquote{semi-continuity} is more common.
			As however the notion of semi-continuity for single-valued functions does not coincide with the notion of hemi-continuity if viewed as a singleton-valued function we use the less common variant in order to be precise.
		}
	if
	the upper preimages of open sets are open.
\end{defn}

The \textdef{graph} of $\Psi$ is given by
\begin{align*}
	\text{Graph}(\Psi) :=& \ml \high{11} (y,x) \in Y \times X \mm x \in \Psi(y) \mr \,.
\end{align*}
Furthermore, we call $\Psi$ \textdef{closed valued} if $\Psi(y) \subseteq X$ is closed for any $y \in Y$.

The Closed Graph Theorem for multivalued functions states that:
\begin{proventhm}[{\cite[Theorem 17.11.]{infinitedimensional}}] 
	\label{thm:closedgraph} 
	The following are equivalent: 
	\begin{enumerate}[a)] 
		\item
			$\mathrm{Graph}(\Psi)$ is closed in the product topology.
		\item
			$\Psi$ is closed-valued and upper hemi-continuous. \qedhere
	\end{enumerate} 
\end{proventhm} 
\begin{cor}
	\label{cor:hemicontinuityofpreimage}
	Let $f: X \to Y$ be continuous.
	Then the multivalued function
	\[ f\inv: Y \longtwoheadarrow X, \; y \longmapsto f\inv\kl \ml y \mr\kr \]
	is upper hemi-continuous.
\end{cor}
\begin{proof}
	As $f$ is continuous, $\Graph(f)$ is closed.
	Note that 
	\[ s: X\times Y \longrightarrow Y\times X, \; (x,y) \longmapsto (y,x) \]
	is a homeomorphism.
	Thus, $\Graph(f\inv) = s(\Graph(f))$ is closed.
	Furthermore, $f\inv\kl \ml y \mr\kr$ is closed, so $f\inv$ is closed-valued.
\end{proof}

This yields an important property of the multivariate diameter of preimages.
\begin{lem}
	\label{lem:diamanduppersemicont}
	If $\Psi: Y \twoheadrightarrow X$ 
	is upper hemi-continuous,
	then
	\[ \diam_m\circ \Psi: Y \longrightarrow \R, \; y \longmapsto \diam_m(\Psi(y))\] 
	is an upper semi-continuous real-valued function for any $m \in \N$.
\end{lem}
\begin{proof}
	Let $\varepsilon > 0$ and let $\delta = \frac{\varepsilon}{2}$.
	The set of all $y' \in Y$ for which
	\[ \Psi(y') \subseteq B_\delta(\Psi(y)) \]
	is an open neighbourhood of $y$.
	Now let $a_1',\ldots,a_m' \in \Psi(y')$.
	There must be $a_1,\ldots,a_m \in \Psi(y)$ with $d(a_i,a_i') < \delta$.
	Now by \reflem{densenessanddiam} we have
	\begin{align*}
		\Dm(a_1',\ldots,a_m') &\leqslant \Dm(a_1,\ldots,a_m) + 2 \delta \\
		&\leqslant \diam_m(\Psi(y)) + 2 \delta = \diam_m(\Psi(y)) + \varepsilon \,.
	\end{align*}
	Taking the supremum we get $\diam_m(\Psi(y')) < \diam_m(\Psi(y))+\varepsilon$ for any $y'$ in an open neighbourhood of $y$.
	As $\varepsilon$ was arbitrary, this proves the upper semi-continuity.
\end{proof}

\subsection{Degree of Denseness}
We define the \textdef{denseness} of a compact $A \subseteq X$ as
\[ \denseness(A) := \inf\ml \varepsilon>0 \mm B_\varepsilon(A) = X \mr = d_\Hc(A,X) \,.\]
\begin{prop}
	\label{prop:continuitydenseness}
	If $K\subseteq G$ is compact, then
	\[ \kappa_K: X \longrightarrow \R^+_0,\; x \longmapsto \sup\ml d(y,\alpha(K,x)) \mm y \in X \mr = \denseness(\alpha(K,x)) \]
	is continuous.
\end{prop}
\begin{proof}
	The Arzelà–Ascoli \refthm{ascoli} implies that $\ml \alpha(k,\cdot) \mm k \in K \mr$
	is equicontinuous.
	Thus, for any $x \in X$ and $\varepsilon > 0$ there is $\delta > 0$ such that
	$\alpha(k,y) \subseteq B_\varepsilon(\alpha(k,x))$ for any $k \in K$ and $y \in B_\delta(x)$.
	Therefore, $\alpha(K,y) \subseteq B_\varepsilon(\alpha(K,x))$ for any $y \in B_\delta(x)$.
	Hence, 
	\[ F_K : X \longrightarrow (\Kc(X),d_\Hc), \;x \longmapsto \alpha(K,x) \]
	is continuous.
	Therefore, $x \mapsto \kappa_K(x) = d_\Hc(F_K(x),X)$ is continuous.
\end{proof}

\begin{prop}
	\label{prop:densenesspointwisetozerominimality}
	Let $(X,\alpha,G)$ be minimal.
	Let $(K_n)_{n\in \N} \in \Pot G^\N$ be an increasing sequence of
	subsets of the group $G$.
	If $\bigcup_{n\in\N} K_n = G$
	then $(\kappa_{K_n})_{n\in\N}$ converges monotonically and pointwise to zero.
\end{prop}
\begin{proof}
	Let $\varepsilon > 0$ and $x\in X$.
	In order to cover $X$, choose finitely many $x_1,\ldots, x_n \in X$ such that
	\[ \bigcup_{i=1}^n B_{\frac \varepsilon 2}(x_i) = X \,.\]
	As $(X,\alpha,G)$ is minimal, $\alpha(G,x)$ is dense.
	For each $i\in \ml 1,\ldots, n\mr$ there is an element $g_i\in G$
	such that $\alpha(g_i,x) \in B_{\frac \varepsilon 2}(x_i)$.
	As $\bigcup_{n\in\N} K_n = G$, there is $N \in \N$
	such that $\ml g_1, \ldots, g_n \mr \subseteq K_n$ for any $n > N$.
	Thus, $\kappa_{K_n}(x) \leqslant \varepsilon$.
	This shows the pointwise convergence. 
	The convergence is monotone due to the monotonicity of $(K_n)_{n\in\N}$.
\end{proof}
\begin{provencor}
	\label{cor:uniformconvergencedenseness}
	By {Dini}'s \refthm{dini} and \refprop{continuitydenseness}, $(\kappa_{K_n})_{n\in\N}$ converges uniformly to zero if the $K_n$ are compact.
\end{provencor}

With these basic results established, we now turn to the analysis of the maximal equicontinuous factor and its role in characterising dynamical properties.

\section{The Maximal Equicontinuous Factor}
\label{sec:MEF}
A topological dynamical system $(Y,\beta,G)$ is \textdef{equicontinuous}
if the family
\[ \ml \beta(g,\cdot) : X \to X \mm g \in G \mr \]
is equicontinuous, \ie
if for any $\varepsilon > 0$ there is a $\delta > 0$ such that 
$d_Y(y,y')<\delta$ implies $d_Y(\beta(g,y),\beta(g,y'))<\varepsilon$ for all $g \in G$.
This is equivalent to the existence of a topologically equivalent $\beta$-invariant metric on $Y$.

An equivalent characterization of minimal equicontinuous systems is given by
\begin{thm}
	\label{thm:classification}
	Let $(Y,\beta,G)$ be minimal.
	$(Y,\beta,G)$ is equicontinuous if and only if there is a compact abelian group $K$, a continuous group homomorphism $\Psi: G \to K$
	and a homeomorphism $\Phi: Y \to K$ such that
	$\Phi(\beta(g,y)) = \Psi(g)\Phi(y)$.
\end{thm}
An abstract and modern presentation of this classical fact can be found in \cite[Theorem 5.4]{halmosvonneumanngroups}.
\refthm{classification} and the uniqueness of the Haar measure from \refthm{haar} together imply
\begin{provencor}
	\label{cor:uniqueergodicityofminimalequicontinuous}
	Minimal and equicontinuous systems are uniquely ergodic.
\end{provencor}

Among all equicontinuous factors of $(X,\alpha,G)$, there is one that is maximal.
This is made precise in the following.
\begin{defn}[Maximal Equicontinuous Factor]
	\label{def:mef}
	A topological dynamical system $(Y,\beta,G)$ together with a factor map $\pi: (X,\alpha,G) \to (Y,\beta,G)$ is a
	\textdef{maximal equicontinuous factor} (MEF) of $(X,\alpha,G)$
	if
	$(Y,\beta,G)$ is equicontinuous and
	for any equicontinuous system $(Z,\gamma,G)$
	and each factor map $p: (X,\alpha,G) \to (Z,\gamma,G)$
	there is a unique factor map $q: (Y,\beta,G) \to (Z,\gamma,G)$ such that $p = q \circ \pi$.
	\begin{figure}[ht]
		\centering
		\catcode`\"=12
		\begin{tikzcd}
			(X,\alpha,G) \arrow[rrdd, "p"] \ar[dd, "\pi"]  & & \\
			& & \\
			 (Y,\beta,G) \arrow[rr, "\exists!\, q", dotted]                &  & (Z,\gamma,G)
		\end{tikzcd}
		\caption{Commutative Diagram of MEF}
		\label{fig:universalpropertyMEF}
	\end{figure}
\end{defn}

While the factor map $\pi$ plays a fundamental role in Definition \ref{def:mef},
it is common practice to notationally suppress it and call $(Y,\beta,G)$ a MEF.
This is justified because any other factor map $\pi': (X,\alpha,G) \to (Y,\beta,G)$ also satisfies
Figure \ref{fig:universalpropertyMEF}.
In fact, there is a conjugacy $a: (Y,\beta,G) \to (Y,\beta,G)$ such that $\pi' = a \circ \pi$.

\begin{proventhm}[{\cite[Chapter 9, Theorem 1]{auslander}}]
	\label{thm:existencemef}
	Every tds $(X,\alpha,G)$ has a MEF $(Y,\beta,G)$ and it is unique up to conjugacy.
\end{proventhm}
%\begin{proof}[Proof of uniqueness]
	%As $\rho$ is a factor map into an equicontinuous system there must be a unique map
	%$\eta: (Y,\beta,G) \to \Zb$ such that $\eta \circ \pi = \rho$.
	%Similarly there must be a unique $\overline \eta: \Zb \to (Y,\beta,G)$ such that $\overline \eta \circ \rho = \pi$.
	%Clearly $(\overline \eta \circ \eta) \circ \pi = \pi$ and $\Id_Y \circ \pi = \pi$.
	%By the uniqueness in the Universal Property \ref{fig:universalpropertyMEF} 
	%we learn that $\overline \eta \circ \eta = \Id_Y$.
	%Similarly $\eta \circ \overline \eta = \Id_\Zb$.
	%So $\eta$ is indeed a conjugacy.
%\end{proof}

As the MEF captures all equicontinuous behavior, it is intuitive that within the fibers the relative dynamics should be sensitive in some sense.
This idea is formalised in the regionally proximal relation and the characterisation of the fibers by it.

Before studying the more complicated \textit{regionally} proximal relation let us define the proximal relation.
\begin{defn}[Proximal Relation]
	Let $m \in \N$ and define the \textdef{proximal $m$-relation}
	\[ \Pc_m := \ml (x_1,\ldots,x_m) \in X^m \mm \inf_{g \in G} D_m \kl \alpha(g,x_1),\ldots,\alpha(g,x_m)\kr = 0 \mr \,. \]
\end{defn}
Those are exactly the tuples that get arbitrarily close together simultaneously.

The regionally proximal relation is a coarser relation that requires the tuple to be arbitrarily well approximated by tuples
whose points simultaneously get as close together as one desires.
Let $\Delta_m$ be the diagonal in $X^m$ and let $\Uf(\Delta_m)$ be the system of its neighbourhoods.
\begin{defn}[Regionally Proximal Relation]
	Let $m \in \N$ and define the \textdef{regionally proximal $m$-relation}
	\[ \Qc_m := \bigcap \ml \overline{\alpha(G,V)} \mm V \in \Uf(\Delta_m) \mr \,.\]
\end{defn}
\begin{rem}
	Let $x_1,\ldots,x_m \in X$.
	Assume that $G$ is first countable.
	Then $(x_1,\ldots,x_m) \in \Qc_m$ if and only if 
	there are sequences $(x_{i,n})_{n\in \N} \in X^\N$ for $i \in \ml 1,\ldots, m \mr$ as well as $(g_n)_{n\in \N} \in G^\N$ 
	with $\lim_{n \to \infty} x_{i,k} = x_i$ and 
	\[ \lim_{n\to\infty} d(\alpha(g_n,x_{i,n}),\alpha(g_n,x_{j,n})) = 0 \] 
	for $i,j \in \ml 1,\ldots, m\mr$.
	In particular we have $\Pc_m \subseteq \Qc_m$.
	If $G$ is not first countable, we must replace sequences by nets.
	In the following we will keep using sequences for simplicity of notation but observe that all arguments work verbatim for nets as well.
\end{rem}
Observe that the action $\alpha$ extends to a diagonal action on $X^m$ by
\[ \alpha^{\otimes m}: G \times X^m \Longrightarrow X^m, \; (g,(x_1,\ldots,x_m)) \longmapsto (\alpha(g,x_1),\ldots,\alpha(g,x_m)) \,. \]

The following characterization resembles notions of sensitivity.
\begin{lem}
	\label{lem:sensitivitycharacterizationofregionallyproximal}
	Let $(X,\alpha,G)$ be minimal.
	Then $(x_1,\ldots,x_m) \in \Qc_m$
	if and only if
	for any $B \subseteq X$ with non-empty interior
	and any neighbourhoods $U_1,\ldots,U_m$ of $x_1,\ldots,x_m$ there is $h \in G$ such that
	$\alpha(h,B) \cap U_i \neq \emptyset$ for $i \in \ml 1, \ldots m \mr$.
\end{lem}
\begin{proof}
	\begin{description}
		\item[$\Longrightarrow$]
			Let $U_1,\ldots,U_m$ be neighbourhoods of $x_1,\ldots,x_m$ and
			\[ (x_1,\ldots,x_m) \in \Qc_m = \bigcap \ml \overline{\alpha(G,V)} \mm V \in \Uf(\Delta_m) \mr \,.\]
			Let $B \subseteq X$ be open.
			As $\alpha(g,\cdot)$ is a homeomorphism for any $g \in G$ the set
			\[ \alpha(G,B) = \bigcup_{g \in G} \alpha(g,B) \]
			is open and invariant.
			By minimality, we have $\alpha(G,B) = X$.

			Now $\alpha^{\otimes m}(G,B^m) \subseteq X^m$ is a neighbourhood of the diagonal.
			So 
			\[ (x_1,\ldots,x_m) \in  \overline{\alpha^{\otimes m}(G,\alpha(G,B^m))} =  \overline{\alpha^{\otimes m}(G,B^m)} \,. \] 
			As $U_1\times \ldots \times U_m$ is a neighbourhood of $(x_1,\ldots,x_m)$ we conclude
			\[  U_1\times \ldots \times U_m \cap \alpha^{\otimes m}(G,B^m) \neq \emptyset \,. \] 
			So there is $g \in G$ such that $\alpha(g,B) \cap U_i \neq \emptyset$ for $i \in \ml 1,\ldots,m \mr$.
			
		\item[$\Longleftarrow$]
			Let $(x_1,\ldots,x_m)$ be such that
			for any $B \subseteq X$ with non-empty interior
			and any neighbourhoods $U_1,\ldots,U_m$ of $x_1,\ldots,x_m$ there is $h \in G$ such that
			$\alpha(h,B) \cap U_i \neq \emptyset$ for $i \in \ml 1, \ldots m \mr$.
			We must show that $(x_1,\ldots,x_m)\in \Qc_m$, \ie
			\[ (x_1,\ldots,x_m) \in \overline{\alpha^{\otimes m}(G,V)} \]
			for any neighbourhood $V$ of the diagonal in $X^m$.
			So let $V \in \Uf(\Delta_m)$ and let $x \in X$.
			There is $\varepsilon > 0$ such that 
			$B_\varepsilon(x)^m \subseteq V$.
			Let $B := B_\varepsilon(x)$.
			We will show that $(x_1,\ldots,x_m) \in \overline{\alpha^{\otimes m}(G,B^m)}$.
			For any neighbourhood $U$ of $(x_1,\ldots,x_m)$
			we must show that $U \cap \alpha^{\otimes m}(G,B^m) \neq \emptyset$.
			By construction of the product topology, there are neighbourhoods $U_i$ of $x_i$ for $i \in \ml 1,\ldots,m\mr$,
			such that $U_1\times \ldots \times U_m \subseteq U$.
			By assumption, there is $h \in G$ such that $\alpha(h,B) \cap U_i \neq \emptyset$ for any $i \in \ldots \ml 1,\ldots, m \mr$.
			Therefore, 
			\[\alpha^{\otimes m}(G,V) \cap U \supseteq \alpha^{\otimes m}(h,B^m) \cap \times_{i=1}^m U_i \neq \emptyset \,. \qedhere\]
	\end{description}
\end{proof}

It is a classical fact that
\begin{proventhm}[{\cite[Chapter 9, Theorem 3]{auslander}}]
	Let $\Qc^*$ be the smallest $\alpha$-invariant, closed equivalence relation containing $\Qc_2$.
	Then $(X/{\Qc^*},\alpha/{\Qc^*}, G)$ is a MEF and so $(x_1,x_2) \in \Qc^*$ if and only if $\pi(x_1) = \pi(x_2)$
	where $\pi: (X,\alpha,G) \to (Y,\beta,G)$ is the factor map to the MEF.
\end{proventhm}

Interestingly, in our setting $\Qc_m$ exactly consists of those tuples that lie in the same fiber of the MEF.
In other words:
\begin{proventhm}[{\cite[Theorem 8]{auslandermultivariateregionallyproximal}}]
	\label{thm:multivariateregionallyproximal}
	Let the topological dynamical system $(X,\alpha,G)$ be minimal and $x_1,\ldots,x_m \in X$.
	Then $(x_1,x_j) \in \Qc_2$ for all $j \in \ml 2,\ldots, m\mr$ if and only if
	$(x_1,\ldots,x_m) \in \Qc_m$.
	Moreover, it holds that $\Qc_2 = \Qc^*$.
\end{proventhm}
This shows that higher-order regionally proximal behavior is already determined by pairwise conditions,
reflecting a certain simplicity in the structure of the MEF.

In our setting, $G$ is abelian, so the conditions required by Auslander’s theorem \cite[Theorem 8]{auslandermultivariateregionallyproximal} are automatically satisfied (as can be seen by \cite[Lemma 5 and the preceding remark]{auslandermultivariateregionallyproximal}).
Non-abelian groups introduce more complications not addressed here.

\section{Invertibility Properties}
\label{sec:invertibility}
For the entire section, let $\pi: (X,\alpha,G) \to (Y,\beta,G)$ be a factor map.
Let $d_X$ and $d_Y$ denote the metrics on $X$ and $Y$ respectively.
Recall the definition of the $m$-diameter
\[ \diam_m(A) = \sup\ml  \min_{1\leqslant i<j\leqslant m} d(a_i,a_j) \mm (a_1,\ldots,a_m) \in A^m \mr \,.\] 
Consider the map that assigns each fiber its $m$-diameter, which is given by
\[ \pi_*\diam_m: Y \longrightarrow \il 0, \diam_m(X) \ir, \; y \longmapsto \diam_m\kl\pi\inv\ml y\mr\kr \,.\]

\refcor{hemicontinuityofpreimage} and \reflem{diamanduppersemicont} together imply
\begin{provenlem}
	\label{lem:pidiam}
	$\pi_*\diam_m$ is upper semi-continuous.
\end{provenlem}

The study of the existence and structure of injectivity points of the factor map plays the key role in the theory of frequent stability and diam-mean equicontinuity.
Understanding where the factor map behaves like an $m$-to-1 map is crucial to classifying more delicate dynamical behaviors,
as we will see in the results that follow.
We generalise the concept of injectivity points to
\begin{defn}[$m$-jectivity]
	We call $x\in X$ an \textdef{$m$-jectivity point}
	if and only if
	\[ \card\kl \pi\inv\kl\ml \pi(x) \mr\kr \kr \leqslant m \,. \]
\end{defn}

Let $\Ic_m$ be the set of $m$-jectivity points and
\[ \Jc_m := \ml y \in Y \mm \card\kl \pi\inv(\ml y\mr)\kr \leqslant m \mr \,.\]
We can switch between both perspectives seamlessly, as
\begin{lem}
	\label{lem:injectivitypointsonbothsides}
	$\Ic_m = \pi\inv(\Jc_m)$.
\end{lem}
\begin{proof}
	If $x \in X$ is an $m$-jectivity point then $y=\pi(x)$ has at most $m$ preimages.
	Thus, $\Ic_m \subseteq \pi\inv(\Jc_m)$.
	Conversely, note that all $k$ preimages of a point $y\in \Jc_m$ are $m$-jectivity points.
	Thus, $\pi\inv(\Jc_m) \subseteq \Ic_m$.
\end{proof}

\begin{lem}
	\label{lem:mjecitvityanddiam}
	\label{lem:gdeltaofinjectivitypoints}
	Let $x \in X$. 
	Then $x \in \Ic_m$ if and only if
	$\pi_*\diam_{m+1}(\pi(x))= 0$.
	Moreover, $\Ic_m$ and $\Jc_m$ are $G_\delta$ sets, i.e.~countable intersections of open sets.
\end{lem}
\begin{proof}
	For the first part we show $\diam_{m+1}(C) = 0$ if and only if $\card(C) \leqslant m$.
	If $C$ has fewer than or equal to $m$ points,
	then you can not form a $(m+1)$-tuple of distinct points,
	hence $\diam_{m+1}(C)=0$.
	Conversely, if $C$ has least $(m+1)$ points you can use their positive minimal pairwise distance to bound $\diam_{m+1}(C)$ from below.

	By Lemma \ref{lem:pidiam}, $\pi_*\diam_{m+1}$ is upper semi-continuous.
	So
	\begin{align*}
		\Jc_m &= \pi_*\diam_{m+1}\inv\kl \ml 0 \mr \kr 
		= \diam_{m+1}\inv\kl \bigcap_{n\in \N} \pi_*\Big[0,\frac 1n\Big) \kr 
		= \bigcap_{n\in \N} \pi_*\diam_{m+1}\inv\kl \Big[0,\frac 1n\Big) \kr
	\end{align*}
	is a countable intersection of open sets.
	So $\Ic_m = \pi\inv(\Jc_m)$ is $G_\delta$.
\end{proof}

\begin{lem}
	\label{lem:diampreimageofball}
	Let $m \in \N$.
	If $x_0 \in \Ic_m$
	then for any $\varepsilon > 0$ there is an $\eta > 0$ such
	that $\diam_{m+1}\kl \pi\inv\kl B_\eta(y) \kr \kr < \varepsilon$ for any $y \in B_\eta(\pi(x_0))$.
\end{lem}
\begin{proof}
	Let $y_0 := \pi(x_0)$.
	Note that the multi-valued map
	\[ D_\rho: Y \longtwoheadarrow Y,\; y \longmapsto \ml y' \in Y \mm d_Y(y,y') \leqslant \rho \mr \]
	has a closed graph for any $\rho \geqslant 0$ by continuity of $d_Y$.
	\refthm{closedgraph} implies that $D_\rho$ is upper hemi-continuous.
	In turn $\pi\inv\circ D_\rho$ is upper hemi-continuous.
	Finally, by \reflem{diamanduppersemicont}, $\diam_{m+1} \circ \pi\inv \circ D_\rho$ is upper semi-continuous.
	The same logic yields the upper semi-continuity of the map $\diam_{m+1} \circ \pi\inv \circ D_\cdot(y_0)$, where
	\[ D_\cdot(y_0) : \R^+_0 \longtwoheadarrow Y, \; \eta \longmapsto D_\eta(y_0) \,.\]
	Thus, $\ml \eta \geqslant 0 \mm \diam_{m+1}\kl \pi\inv\kl D_\eta(y_0) \kr \kr < \varepsilon\mr$
	is open.
	It contains $0$, as $y_0$ has less than $m$ preimages which implies 
	\[ \diam_{m+1}\kl \pi\inv\kl D_0(y) \kr\kr = \diam_{m+1}\kl \pi\inv\kl \ml y \mr \kr\kr = 0 \,. \]
	Therefore, there is $\eta_1 > 0$ 
	such that
	$\diam_{m+1}\kl \pi\inv\kl D_{\eta_1}(y_0) \kr \kr < \varepsilon$.
	So
	\[ y_0 \in \ml y \in Y \mm \diam_{m+1}\kl \pi\inv\kl D_{\eta_1}(y) \kr \kr < \varepsilon \mr =: V \]
	As $y \mapsto \diam_{m+1} \circ \pi\inv \circ D_{\eta_1}(y)$ is upper semi-continuous, $V$ is open.
	Hence, there is $\eta_2>0$ such that
	$\diam_{m+1}\kl \pi\inv\kl D_{\eta_1}(y) \kr \kr< \varepsilon$ for all $y \in B_{\eta_2}(y_0)$.
	By monotonicity, $\eta : = {\min(\eta_1,\eta_2)} $ yields
	$\diam_{m+1}\kl \pi\inv\kl B_{\eta}(y) \kr \kr < \varepsilon$ 
	for any $y \in B_\eta(y_0)$.
\end{proof}

Let us now study specific ways in which a factor map can be finite-to-one.
\subsection{Almost Surely Finite-to-One}
Assume that $(Y,\beta,G)$ is uniquely ergodic with invariant measure $\nu$.
\begin{defn}
	We call $\pi$ \textdef{almost surely $m$:1} if $\nu(\Jc_m) = 1$.
\end{defn}

A direct consequence of \reflem{injectivitypointsonbothsides} is
\begin{provenlem}
	Let $(X,\alpha,G)$ be uniquely ergodic with unique invariant measure $\mu$.
	Then $\pi$ is almost surely $m$:1 if and only if $\mu(\Ic_m) = 1$.
\end{provenlem}

\begin{lem}[Dichotomy]
	\label{lem:dichotomyalmostsurely11}
	Let $(X,\alpha,G)$ be minimal.
	Then {either}
	\[ \pi \text{ is almost surely $m$:1} \quad{or}\quad \exists \varepsilon > 0: \nu\kl \ml y \in Y \mm \pi_*\diam_{m+1}(y) > \varepsilon \mr \kr > \varepsilon \,.\]
\end{lem}
\begin{proof}
	Assume that $\pi$ is not almost surely $m$:1.
	By \reflem{mjecitvityanddiam},
	\[ \nu\kl \ml y \in Y \mm \pi_*\diam_{m+1}(y) > 0 \mr \kr > 0 \,. \]
	By $\sigma$-additivity, there must be some $n \in \N$ such that
	\[ c := \nu\kl \ml y \in Y \mm \pi_*\diam_{m+1}(y) > \frac{1}{n} \mr \kr > 0 \,. \]
	Setting $\varepsilon = \min(\frac{1}{n},c)$ finishes the proof.
\end{proof}

The property \enquote{almost surely finite-to-one} focuses on the measure-theoretic behaviour of $\pi$.
We now turn to a weaker topological condition, called \enquote{almost $m$:1}, which depends on the residual nature of $\Ic_m$.
\subsection{Almost Finite-to-One}
\begin{defn}
	We call $\pi$ \textdef{almost $m$:1} if $\Ic_m$ is residual.
\end{defn}

\begin{prop}
	\label{prop:minimalitygenerically1-1}
	Let $(X,\alpha,G)$ be minimal.
	If $\pi$ has at least one $m$-jectivity point, 
	then $\pi$ is almost $m$:1.
\end{prop}
\begin{proof}
	Assume that $\Ic_m$ is non-empty.
	Observe that $\Ic_m$ is invariant because
	\[ \pi\inv\kl \pi(\ml \alpha(g,x) \mr ) \kr = \pi\inv\kl \ml \beta(g,\pi( x)) \mr \kr = \alpha \kl g, \pi \inv \ml \pi(x) \mr \kr \,. \]
	By minimality, $\Ic_m$ is dense and
	by \reflem{gdeltaofinjectivitypoints}, it is $G_\delta$.
\end{proof}
\begin{provencor}
	\label{prop:generically1-1}
	Let $(X,\alpha,G)$ be minimal.
	$\Jc_m$ is non-empty
	if and only if
	$\pi$ is almost $m$:1.
\end{provencor}

The following generalises {\cite[Lemma 2.4]{isoext}}.
\begin{lem}[Dichotomy]
	\label{lem:dichotomygenerically11}
	Let $(X,\alpha,G)$ be minimal.
	Then {either}
	\[\pi \text{ is almost $m$:1} \quad{or}\quad \inf_{y\in Y} \pi_*\diam_{m+1}(y) > 0 \,.\]
\end{lem}
\begin{proof}
	If $\pi$ is almost $m$:1, 
	then there is an element $y_0 \in \Jc_m$.
	Thus, 
	\[ \inf_{y\in Y}\pi_*\diam_{m+1}(y) = 0 \,. \]
	Hence, both cases of the dichotomy are disjoint.

	Assume that $\inf_{y\in Y}\pi_*\diam_{m+1}(y) = 0$.
	By standing assumption $G$ is $\sigma$-compact.
	Pick an increasing sequence of compact subsets $K_n\subseteq G$ such that $G = \bigcup_{n\in\N} K_n$.
	By Arzelà-Ascoli \refthm{ascoli}, for any $\varepsilon > 0$ there is $\delta > 0$ such that
	\[  d_X(x,x')<\delta \Longrightarrow d_X(\alpha(k,x),\alpha(k,x'))<\varepsilon \]
	for any $n \in \N$ and $k \in K_n$.
	Use this in order to recursively construct a sequence $\kl\delta_n\kr_{n\in\N} \in (0,1)^\N$ with $\lim_{n\to\infty} \delta_n = 0$ such that
	\begin{align}
		\label{eq:deltastep}
		d_X(x,x') < \delta_{n+1} \Longrightarrow d_X\kl \alpha(k,x),\alpha(k,y) \kr < \delta_n
	\end{align}
	for all $n\in \N$, all $x,x'\in X$ as well as all $k \in K_n$.
	Consider the level sets 
	\[ A_n := \ml y \in Y \mm \pi_*\diam_{m+1}(y) < \delta_n \mr \,. \]
	
	We claim that $\beta(k,y) \in A_n$ for any $y \in A_{n+1}$ and $k \in K_n$.
	In order to determine $\diam_{m+1}\kl \pi\inv\kl \ml \beta(k,y) \mr\kr \kr$
	pick $x_1,\ldots,x_{m+1} \in \pi\inv\kl \ml \beta(k,y) \mr\kr$.
	Clearly, $\pi\inv\kl\ml \beta(k,y)\mr\kr =  \alpha\kl k,\pi\inv\kl \ml y\mr \kr\kr$.
	So there must be $x_1',\ldots,x_{m+1}' \in \pi\inv\kl \ml y \mr \kr$
	such that $\alpha(k,x_i') = x_i$ for $i \in \ml 1,\ldots,m+1 \mr$.
	By the assumption that $y \in A_{n+1}$, there is $i \neq j$ such that $d_X(x_i',x_j') < \delta_{n+1}$.
	By choice of $\delta_{n+1}$ we have $d_X(x_i,x_j) < \delta_n$.
	This shows that $\beta(k,A_{n+1}) \subseteq A_n$ for any $k \in K_n$.

	As $\pi_*\diam_{m+1}$ is upper semi-continuous,
	the level sets $A_n$ are open.
	The function $x \mapsto \kappa_{K_n} = \mathrm{denseness}(\alpha(K_n,x))$
	converges uniformly to zero by \refcor{uniformconvergencedenseness}.

	Now let $\varepsilon > 0$.
	Fix $k \in \N$ such that $\kappa_{K_k} < \varepsilon$ uniformly.
	The assumption that $\pi_*\diam_{m+1}$ has no positive infimum implies that the $A_n$ are non-empty.
	As the level sets are mapped into each other,
	$\alpha(K_k,x) \subseteq A_n$ for any $x \in A_{n+k}$.
	As $\kappa_{K_k} < \varepsilon$ we have 
	\[ \denseness(A_n) \leqslant \denseness(\alpha(K_k,x)) < \varepsilon \,. \]
	As $\varepsilon$ was arbitrary, $\denseness(A_n) = 0$, \ie $A_n$ is dense.
    So $\Ic_m = \bigcap_{n\in \N} A_n$ is residual
	and $\pi$ is almost $m$:1.
\end{proof}

\section{Multivariate Frequent Stability}
\label{sec:frequentlystable}
Here, we define and investigate multivariate versions of the notion of frequent stability, originally introduced in \cite{diammeanequicts}.
Thereby arguments from \cite[Section 3]{diammeanequicts} and \cite{frequentlystableisalmostautomorphic} are generalised
of which also inspiration was taken.

The idea of (multivariate) frequent stability is that, \emph{frequently}, the iterate of a sufficiently small ball has a small (multivariate) diameter.
Define 
\[ T_{\delta,\varepsilon}^{(m)}(x):= \ml g \in G\mm \diam_m\kl \alpha(g, B_\delta(x))\kr > \varepsilon \mr \,. \]

Let $\Fc = (F_n)_{n\in\N}$ be a F\o lner sequence and $x \in X$.
\begin{defn}[Frequent $m$-Stability]
	The point $x$ is \textdef{Banach $\Fc$-frequently $m$-stable}
	if for any $\varepsilon > 0$ there is $\delta > 0$ with 
	\[ \BDu_\Fc\kl T_{\delta,\varepsilon}^{(m)}(x) \kr< 1 \,. \]
	The point $x$ is \textdef{weakly $\Fc$-frequently $m$-stable}
	if for any $\varepsilon > 0$ there is $\delta > 0$ with
	\[ \BDl_\Fc\kl T_{\delta,\varepsilon}^{(m)}(x) \kr< 1 \,. \]
\end{defn}

The compactness of $X$ allows us to deduce the uniform version of Banach frequent $m$-stability from the pointwise:
\begin{lem}
	\label{lem:frequentstabilitycompact}
	If every point $x \in X$ is Banach $\Fc$-frequently $m$-stable,
	then for every $\varepsilon > 0$ there is $\delta > 0$ such that
	$\BDu_\Fc(T_{\delta,\varepsilon}^{(m)}(x))< 1$ holds for any $x \in X$.
\end{lem}
\begin{proof}
	Let $\varepsilon > 0$.
	For $x \in X$ let $\rho(x) := \sup D_x$, where $D_x$ is the non-empty set
	\[ D_x := \ml \delta_x > 0 \mm \BDu_\Fc\kl T_{\delta_x,\varepsilon}^{(m)}(x) \kr< 1 \mr \,.\]
	To demonstrate that $\rho: X \to \R^+$ has a minimum,
	it suffices to prove that $\rho$ is lower semi-continuous.
	%or more specifically that 
	%$B_\eta(x_0) \subseteq \ml x \in X \mm \rho(x) \geqslant \rho(x_0) - \eta \mr$
	%for any $x_0 \in X$ and $\eta > 0$.
	Let $x_0 \in X$, $\delta' \in D_{x_0}$,	$0 < \eta < \delta'$ and $x \in B_\eta(x_0)$.
	Then $B_{\delta'-\eta}(x) \subseteq B_{\delta'}(x_0)$.
	In particular,
	$\diam_{m}\kl \alpha\kl g,B_{\delta'-\eta}(x) \kr \kr \leqslant \diam_{m}\kl \alpha\kl g,B_{\delta'}(x_0) \kr \kr$
	for any $g \in G$.
	Therefore, $\delta'-\eta \in D_x$.
	By taking the supremum over $\delta'$, conclude that $\rho(x) \geqslant \rho(x_0) - \eta$.
	So $B_\eta(x_0) \subseteq \rho\inv\kl (\rho(x_0)-\eta, \infty) \kr$.
\end{proof}

Recall that
\[ J = \ml K \in \Kc(X) \mm \tint(K) \neq \emptyset \mr \]
is the family of compact subsets with non-empty interior.
\begin{prop}
	\label{prop:notalmostm1impliesnotbanachfrequentlystable}
	Let $(X,\alpha,G)$ be minimal and let $(Y,\beta,G)$ be the MEF with factor map $\pi: (X,\alpha,G) \to (Y,\beta,G)$
	and let $\Fc = (F_n)_{n\in\N}$ be any F\o lner sequence.

	If $\pi$ is not almost $m$:1
	then there is $\varepsilon > 0$ such that for any $B \in J$ there is $(a_n)_{n\in\N} \in G^\N$ such that the shifted F\o lner sequence $\Fc_B := (F_{B,n})_{n\in\N}$, where $F_{B,n} = a_n F_n$,
	satisfies 
	\[ \limsup_{n\to \infty} \frac{m\kl \ml g \in a_nF_n \mm \diam_{m+1}\kl \alpha(g,B) \kr > \varepsilon \mr \kr}{m(a_nF_n)} = 1 \,. \]
\end{prop}
\begin{proof}
	This argument follows and generalises the proof of \cite[Theorem 3.1]{frequentlystableisalmostautomorphic}.

	Let $B \in J$ be arbitrary.
	By \reflem{dichotomygenerically11}, there is $\varepsilon > 0$ such that
	\[ \diam_{m+1} \kl \pi\inv\kl \ml y \mr \kr \kr > 5\varepsilon \]
	for all $y \in Y$.
	Now let $n \in \N$ be arbitrary.
	Pick any $y_* \in Y$.
	There is $\eta_n > 0$ such that
	\[ d(x,x') < \eta_n \Longrightarrow d(\alpha(g,x),\alpha(g,x')) < \varepsilon \]
	for any $g \in F_n$ and $x,x' \in X$.
	Furthermore, there is a finite set
	\[ \ml x_1, \ldots, x_k \mr \subseteq \pi\inv(\ml y_* \mr) \]
	such that $B_{\eta_n}(\ml x_1,\ldots,x_k \mr) \supseteq\pi\inv(\ml y_* \mr)$.
	Note that this implies in particular that 
	\[ B_{\varepsilon}(\alpha(g,\ml x_1,\ldots,x_k \mr)) \supseteq \pi\inv(\beta(g,\ml y_* \mr))\]
	for any $g \in F_n$.

	By \refthm{multivariateregionallyproximal}, 
	$(x_1,\ldots,x_{k}) \in \Qc_k$.
	\reflem{sensitivitycharacterizationofregionallyproximal} implies that for $n \in \N$ there is $a_n \in G$ and $z_1,\ldots,z_k \in B$ such that $d(x_i,\alpha(a_{n},z_{i})) < \eta_n$ for any $i \in \ml 1,\ldots,k\mr$.
	Therefore, $d(\alpha(g,x_i),\alpha(a_n g, z_{i})) < \varepsilon$
	for any $g \in F_n$ and $i \in \ml 1,\ldots,k\mr$.
	Thus
	\[ B_{2\varepsilon}\kl \alpha\kl a_n g, \ml z_{1} \ldots z_{k}  \mr \kr  \kr\supseteq \pi\inv\kl \ml \beta(a_n g,y_*) \mr \kr \,.\]
	By \reflem{densenessanddiam}, we conclude $\diam_{m+1}\kl \alpha(a_n g, \ml z_{1},\ldots,z_{k} \mr )\kr > \varepsilon$ for any $g \in F_n$.

	Define a shifted F\o lner sequence $\tilde \Fc_B = (\tilde F_n)_{n \in \N}$ with $\tilde F_n = a_nF_n$.
	Then
	\begin{align*}
		\diam_{m+1}\kl \alpha(f,B) \kr &\geqslant \diam_{m+1}\kl \alpha(f, \ml z_{1},\ldots,z_{k} \mr \kr > \varepsilon
	\end{align*}
	for any $n \in \N$ and $f \in \tilde F_n$ and so immediatly
	\[ \limsup_{n\to \infty} \frac{m\kl \ml g \in a_nF_n \mm \diam_{m+1}\kl \alpha(g,B) \kr > \varepsilon \mr \kr}{m(a_nF_n)} = 1 \,. \qedhere \]
\end{proof}
\begin{cor}
	Setting $B := \overline{B_{\frac{\delta}{2}}(x)}$ in \refprop{notalmostm1impliesnotbanachfrequentlystable} yields that $\BDu_\Fc\kl T^{(m)}_{\delta,\varepsilon}(x) \kr = 1$ for any $\delta > 0$ and $x \in X$.
	So if $\pi$ is not almost $m$:$1$, then no $x \in X$ can be Banach $\Fc$-frequently $(m+1)$-stable.
\end{cor}

The following statement, which is the main result of this section,
generalises \cite[Theorem 3.1]{frequentlystableisalmostautomorphic} to the multivariate case.
\begin{thm}
	\label{thm:mainthmfreqstable}
	Let $(X,\alpha,G)$ be minimal and let $(Y,\beta,G)$ be the MEF with factor map $\pi: (X,\alpha,G) \to (Y,\beta,G)$.
	The following are equivalent:
	\begin{enumerate}[(i)]
		\item \label{item:almostm1}
			$\pi$ is almost $m$:1;
		\item \label{item:strongfrequentmstability}
			Every $x\in X$ is Banach $\Fc$-frequently $(m+1)$-stable, for any F\o lner sequence $\Fc$;
		\item \label{item:weakfrequentmstability}
			There exists some F\o lner sequence $\Fc$ and a point $x \in X$ such that
			$x$ is weakly $\Fc$-frequently $(m+1)$-stable.
	\end{enumerate}
			
\end{thm}
\begin{proof}
	The metric $d_Y$ on $Y$ can be assumed to be $\beta$-invariant.
	The unique invariant measure on $(Y,\beta,G)$ will be denoted by $\nu$.
	\begin{description}
		\item[\refi{almostm1} implies \refi{strongfrequentmstability}]
			Let $\varepsilon > 0$, $x \in X$ and $\Fc$ an arbitrary F\o lner sequence.

			By assumption, there is $y_0 \in Y$ such that $\card\kl \pi\inv\kl\ml y_0 \mr\kr \kr \leqslant m$.
			\reflem{diampreimageofball} implies that there is $\eta > 0$ such that
			\begin{align}
				\label{eq:diameterofpreimages}
				\diam_{m+1}\kl \pi\inv\kl B_\eta(y) \kr \kr < \varepsilon
			\end{align}
			for any $y \in B_\eta(y_0)$.
			As $\pi$ is continuous,
			there is $\delta > 0$ such that
			$\pi\kl B_\delta(x) \kr \subseteq B_\eta(\pi(x))$.
			The invariance of $d_Y$ implies
			$B_\eta\kl \beta(g,\pi(x)) \kr = \beta\kl g, B_\eta(\pi(x)) \kr$
			for any $g\in G$.
			Furthermore, the equivariance of $\pi$ yields
			\begin{align*}
				\pi\kl \alpha\kl g,B_\delta(x) \kr \kr &= \beta\kl g, \pi\kl B_\delta(x) \kr \kr \\
				&\subseteq \beta\kl g, B_\eta(\pi(x)) \kr \\
				&= B_\eta\kl \beta(g,\pi(x)) \kr
			\end{align*}
			for all $g \in G$.
			Taking the preimage, we further obtain
			\begin{align}
				\label{eq:subsetsfreqstable}
				\alpha\kl g,B_\delta(x) \kr \subseteq \pi\inv\kl B_\eta\kl \beta( g, \pi(x) ) \kr \kr \,.
			\end{align}
			Combining \refe{eq:subsetsfreqstable} and \refe{eq:diameterofpreimages},
			we see that
			$\diam_{m+1}\kl \alpha\kl g,B_\delta(x) \kr \kr < \varepsilon$, whenever $\beta(g,\pi(x)) \in B_\eta(y_0)$.
			In other words, we have
			\[ H := \ml g\in G \mm \beta(g,\pi(x)) \in B_\eta(y_0) \mr \subseteq \ml g \in G \mm   \diam_{m+1}\kl \alpha\kl g,B_\delta(x) \kr \kr < \varepsilon \mr \,. \]
			A straightforward application of \refcor{uniformergodicandbanach} shows that $H$ has a positive lower Banach $\Fc$-density.

		\item[\refi{strongfrequentmstability} implies \refi{weakfrequentmstability}]
			This is immediate.

		\item[\refi{weakfrequentmstability} implies \refi{almostm1}]
			We proceed by contradiction.
			Fix any F\o lner sequence $\Fc = (F_n)_{n\in\N}$.
			So assume that $\pi$ is not almost $m$:1 and that $x \in X$ is weakly $\Fc$-frequently $(m+1)$-stable.
			Let $\varepsilon > 0$ be according to \refprop{notalmostm1impliesnotbanachfrequentlystable}.

			By \reflem{continuityofdiam}, $\diam_{m+1}$ is continuous, so
			\[ \Omega := \diam_{m+1}\inv\kl \il 0,\varepsilon \kr \kr = \ml K \in \Kc(X) \mm \diam_{m+1}(K) < \varepsilon \mr \]
			is open. 
			Recall that the action of $\alpha$ on $X$ induces an action $\hat\alpha$ on $\Kc(X)$ by
			\[ \hat\alpha(g,K) = \ml \alpha(g,k) \mm k \in K \mr \,. \]
			By assumption on $x$, there is $\delta > 0$ and $(b_n)_{n\in\N} \in G^\N$ such that
			\begin{align}
				\label{eq:contradictionassumptionfrequentstability}
				\lim_{n\to\infty} \frac{\kl \ml g \in b_nF_n \mm \diam_{m+1}(\alpha(g,B_\delta(x))) < \varepsilon \mr \kr}{m(b_nF_n)} > 0 \,.
			\end{align}
			Let $B := \overline{B_{\frac{\delta}{2}}(x)}\subseteq B_\delta(x)$.
			Krylov-Bogolyubov \refthm{krylovbogolyubov} implies that, by passing to a subsequence, there is an $\hat\alpha$-invariant measure $M$ such that
			\[ \lim_{n \to \infty} \frac{1}{m(b_n F_{n})} \int_{b_n F_{n}} \delta_{\hat\alpha(g,B)} \dd m(g) = M \,. \]
			By the Ergodic Decomposition \refthm{ergodicdecomp}, there is a measure $E$ on the set of ergodic measures such that $M = \int M_e \dd E(M_e)$.
			Equation \refe{eq:contradictionassumptionfrequentstability} implies $M(\Omega) > 0$.
			So 
			\[ E\kl \ml M_e \text{ ergodic } \mm M_e(\Omega) > 0 \mr \kr > 0 \,. \]
			Note that, by \refcor{existenceofergodiccomponentsupportedonJ} and \reflem{decompositionofJ}, there is an ergodic measure $M^*_e$ and $B_* \in J$ such that $M^*_e( \ml K \in J \mm B_* \subseteq K \mr) > 0$ 
			and
			\begin{align}
				\label{eq:contradictionstepfrequentstability}
				M^*_e(\Omega) > 0 \,.
			\end{align}
			Let $\Fc_{B_*}$ be the shifted F\o lner sequence given by \refprop{notalmostm1impliesnotbanachfrequentlystable}.
			The Lindenstrauss Ergodic \refthm{lindenstrauss} applied to a tempered subsequence of $\Fc_{B_*}$
			and $\1_{\Omega^c}$
			implies that
			\[ \frac{1}{m(F_{B_*,n})} \int_{F_{B_*,n}} \1_{\Omega^c}\kl \hat\alpha(g,K) \kr \dd m(g) \xlongrightarrow{n\to\infty} M_e^*(\Omega^c) \]
			for almost all $K \in \Kc(X)$.
			As $M^*_e( \ml K \in J \mm B_* \subseteq K \mr) > 0$, we can assume that $K \supseteq B_*$ and this implies
			\begin{align*}
				\label{eq:contradictionfrequentstability}
				M^*_e(\Omega^c)=1 \,.
			\end{align*}
			This yields a contradiction to \refe{eq:contradictionstepfrequentstability}.
			\qedhere
	\end{description}
\end{proof}

\section{Multivariate Diam-Mean Equicontinuity}
\label{sec:diammean}
In \cite{diammeanequicts}, the definition of (univariate) diam-mean equicontinuity was given.
A generalisation to the multivariate case is presented here.
Thereby arguments from \cite[Section 4]{diammeanequicts}
and \cite{frequentlystableisalmostautomorphic} are followed and generalised
of which also inspiration was taken.

The idea of (multivariate) diam-mean equicontinuity is that \emph{almost all} iterates of a sufficiently small ball have a small (multivariate) diameter.
Recall the definition
\[ T_{\delta,\varepsilon}^{(m)}(x) = \ml g \in G\mm \diam_m\kl \alpha(g, B_\delta(x))\kr > \varepsilon \mr \,. \]

The proofs in this section are very analogous to those in Section \ref{sec:frequentlystable}
and, for the ease of the reader written as similar as possible.

Let $\Fc = (F_n)_{n\in\N}$ be a F\o lner sequence and $x \in X$.
\begin{defn}[Diam-Mean $m$-Equicontinuity]
	The point $x$ is \textdef{Banach $\Fc$-diam-mean $m$-equicontinuous} 
	if for all $\varepsilon > 0$ there is $\delta > 0$ such that
	\[ \BDu_\Fc(T_{\delta,\varepsilon}^{(m)}(x))< \varepsilon \]
	The point $x$ is \textdef{weakly $\Fc$-diam-mean $m$-equicontinuous}
	if for all $\varepsilon > 0$ there is $\delta > 0$ such that
	\[ \BDl_\Fc(T_{\delta,\varepsilon}^{(m)}(x))< \varepsilon \]
\end{defn}

\begin{lem}
	A point $x \in X$ is Banach $\Fc$-diam-mean $m$-equicontinuous 
	if and only if
	for all $\varepsilon > 0$ there is $\delta > 0$ such that
	\[ A_\delta(x) := \limsup_{n\to\infty} \sup_{h\in G} \frac{1}{m(hF_n)}\int_{hF_n} \diam_{m}\kl \alpha(g, B_\delta(x) \kr \dd m(g) < \varepsilon \,.\]
\end{lem}
\begin{proof}
	Without loss of generality assume that $\diam_{m+1}(X) = 1$.
	\begin{description}
		\item[$\Longrightarrow$]
			Let $\varepsilon > 0$.
			There is $\delta > 0$ such that
			\[ \BDu_{\Fc} \kl \ml { g \in G } \mm { \diam_{m+1}\kl \alpha(g, B_\delta(x)) \kr > \frac{\varepsilon}{2} } \mr \kr < \frac{\varepsilon}{2} \,.\]
			Let $H := \ml { g \in G } \mm { \diam_{m+1}\kl \alpha(g, B_\delta(x)) \kr > \frac{\varepsilon}{2} } \mr$ and calculate
			\begin{align*}
				A_\delta(x) &= \limsup_{n\to\infty} \sup_{h\in G} \frac{1}{m(F_n)} \int_{F_n \cdot h} \diam_{m+1}\kl \alpha(g,B_\delta(x)) \kr \dd m(g) \\
				&\leqslant \diam_{m+1}(X) \cdot \BDu_{\Fc}\kl H \kr + \frac{\varepsilon}{2} \cdot  \BDu_{\Fc}\kl H^c \kr \\
				&< 1\cdot \frac{\varepsilon}{2} + \frac{\varepsilon}{2} \cdot 1 = \varepsilon \,.
			\end{align*}
		\item[$\Longleftarrow$]
			We proceed by contraposition.
			Assume that there is an $\eta > 0$ such that
			\[ \BDu_{\Fc} \kl \ml { g \in G } \mm { \diam_{m+1}\kl \alpha(g, B_\delta(x)) \kr > \eta } \mr \kr \geqslant \eta \,.\]
			for any $\delta > 0$.
			Then
			\begin{align*}
				A_\delta(x) &=\limsup_{n\to\infty} \sup_{h\in G}\frac{1}{m(F_n)} \int_{F_n \cdot h} \diam_{m+1}\kl \alpha(g,B_\delta(x)) \kr \dd m(g) \\ 
				&\geqslant \eta \cdot \BDu_{\Fc}\kl \ml { g \in G } \mm { \diam_{m+1}\kl \alpha(g, B_\delta(x)) \kr > \eta } \mr \kr \\
				&\geqslant \eta \cdot \eta = \varepsilon
			\end{align*}
			for any $\delta > 0$, where $\varepsilon := \eta^2$. \qedhere
	\end{description}
\end{proof}

The same proof as for \reflem{frequentstabilitycompact} shows that
\begin{provenlem}
	\label{lem:diammeancompact}
	If every point $x \in X$ is Banach $\Fc$-diam-mean $m$-equicontinuous,
	then for every $\varepsilon > 0$ there is $\delta > 0$ such that
	$\BDu_\Fc(T_{\delta,\varepsilon}^{(m)}(x))< \varepsilon$ holds for any $x \in X$.
\end{provenlem}

As before, we let
\[ J := \ml K \in \Kc(X) \mm \tint(K) \neq \emptyset \mr \,. \]

\begin{prop}
	\label{prop:notalmostsurelym1impliesnotbanachdiammean}
	Let $(X,\alpha,G)$ be minimal and let $(Y,\beta,G)$ be the MEF with factor map $\pi: (X,\alpha,G) \to (Y,\beta,G)$
	and let $\Fc = (F_n)_{n\in\N}$ be any F\o lner sequence.

	If $\pi$ is not almost surely $m$:1
	then there is $\varepsilon > 0$ such that for any $B \in J$ there is $(a_n)_{n\in\N} \in G^\N$ such that the shifted F\o lner sequence $\Fc_B := (F_{B,n})_{n\in\N}$, where $F_{B,n} = a_n F_n$,
	satisfies 
	\[ \limsup_{n\to \infty} \frac{m\kl \ml g \in a_nF_n \mm \diam_{m+1}\kl \alpha(g,B) \kr > \varepsilon \mr \kr}{m(a_nF_n)} > 2\varepsilon \,. \]
\end{prop}
\begin{proof}
	This argument follows and generalises the proof of \cite[Theorem 3.1]{frequentlystableisalmostautomorphic}.

	Let $B \in J$ be arbitrary.
	By \reflem{dichotomyalmostsurely11}, there is $\varepsilon > 0$ such that
	\[ \mu\kl \ml y \in Y \mm \pi_*\diam_{m+1} \kl \ml y \mr \kr  > 5\varepsilon \mr \kr > 2\varepsilon \,. \]
	Now let $n \in \N$ be arbitrary.
	By the Lindenstrauss Ergodic \refthm{lindenstrauss}, there is $y_* \in Y$ such that
	\begin{align}
		\label{eq:densitiyofy*}
		\Dl_\Fc\kl \ml g \in G \mm \pi_*\diam_{m+1}\kl \ml \beta(g,y_*) \mr \kr > 5\varepsilon \mr \kr > 2\varepsilon \,.
	\end{align}
	Pick any $y_* \in Y$.
	As in \refprop{notalmostm1impliesnotbanachfrequentlystable} there is $a_n \in G$ such that
	\[ B_{2\varepsilon}\kl \alpha\kl a_n g, B \kr  \kr\supseteq \pi\inv\kl \ml \beta(a_n g,y_*) \mr \kr \]
	for any $g \in F_n$.
	Let $H := \ml g \in G \mm \pi_*\diam_{m+1}\kl \ml \beta(g,y_*) \mr \kr > 5\varepsilon \mr$.
	Now, \refe{eq:densitiyofy*} implies that
	\[ \liminf_{n\to\infty} \frac{m\kl a_nF_n \cap H \kr}{m(a_n F_n)}  > 2\varepsilon \,.\]
	By \reflem{densenessanddiam}, we conclude $\diam_{m+1}\kl \alpha(h, B) \kr > \varepsilon$ for $h \in a_nF_n \cap H$.
	Define a shifted F\o lner sequence $\tilde \Fc_B = (\tilde F_n)_{n \in \N}$ with $\tilde F_n = a_nF_n$.
	Then
	\begin{align*}
		\diam_{m+1}\kl \alpha(f,B) \kr > \varepsilon
	\end{align*}
	for any $n \in \N$ and $f \in \tilde F_n \cap H$ and so immediatly
	\[ \limsup_{n\to \infty} \frac{m\kl \ml g \in a_nF_n \mm \diam_{m+1}\kl \alpha(g,B) \kr > \varepsilon \mr \kr}{m(a_nF_n)} > 2\varepsilon \,. \qedhere \]
\end{proof}
\begin{cor}
	Setting $B := \overline{B_{\frac{\delta}{2}}(x)}$ in \refprop{notalmostsurelym1impliesnotbanachdiammean} yields that $\BDu_\Fc\kl T^{(m)}_{\delta,\varepsilon}(x) \kr > \varepsilon$ for any $\delta > 0$ and $x \in X$.
	So if $\pi$ is not almost surely $m$:$1$, then no $x \in X$ can be Banach $\Fc$-diam-mean $(m+1)$-equicontinuous.
\end{cor}

The following statement, which is the main result of this section,
generalises \cite[Theorem 4.12]{diammeanequicts} to the multivariate case.
\begin{thm}
	\label{thm:mainthmdiamequi}
	Let $(X,\alpha,G)$ be minimal and let $(Y,\beta,G)$ be the MEF with factor map $\pi: (X,\alpha,G) \to (Y,\beta,G)$.
	The following are equivalent:
	\begin{enumerate}[(i)]
		\item \label{item:almostsurelym1}
			$\pi$ is almost surely $m$:1
		\item \label{item:strongdiammean}
			Every $x\in X$ is Banach $\Fc$-diam-mean $(m+1)$-equicontinuous, for any F\o lner sequence $\Fc$.
		\item \label{item:weakdiammean}
			There exists some F\o lner sequence $\Fc$ and a point $x \in X$ such that
			$x$ is weakly $\Fc$-diam-mean $(m+1)$-equicontinuous.
	\end{enumerate}
\end{thm}
\begin{proof}
	The metric $d_Y$ on $Y$ can be assumed to be $\beta$-invariant.
	The unique invariant measure on $(Y,\beta,G)$ will be denoted by $\nu$.
	\begin{description}
		\item[\refi{almostsurelym1} implies \refi{strongdiammean}]
			Let $\varepsilon > 0$, $x \in X$ and $\Fc$ an arbitrary F\o lner sequence.
			Define
			\[ A := \ml y \in Y \mm { \card\kl \pi\inv\kl \ml y \mr \kr \kr \leqslant m } \mr \,.\]
			As $\pi$ is almost surely $m$:1, $\nu(A)=1$.
			By regularity of $\nu$, there is a compact $K \subseteq A$ such that
			$\nu(K) > 1-\varepsilon$.
			By continuity of $d_Y$, the multi-valued map
			\[ D_\cdot(y) : \R^+_0 \longtwoheadarrow Y, \; \eta \longmapsto \ml y' \in Y \mm d_Y(y,y') \leqslant \eta \mr \,.\]
			has a closed graph for any $y \in Y$.
			\refthm{closedgraph} implies that $D_\cdot(y)$ and thus $\pi\inv\circ D_\cdot(y)$ is upper hemi-continuous.
			Therefore,
			\[ E_y := \ml \eta \geqslant 0 \mm \pi\inv\kl D_\eta(y) \kr \subseteq B_\frac{\varepsilon}{2}\kl \pi\inv\kl \ml y \mr \kr\kr\mr \]
			is open for $y \in Y$.
			As $0 \in E_y$, for all $y \in Y$, there is $\eta_y> 0$ such that
			\[ \pi\inv\kl B_{2\eta_y}(y) \kr \subseteq B_\frac{\varepsilon}{2}\kl \pi\inv\kl \ml y \mr \kr \kr \,.\]

			Clearly, $\ml B_{\eta_y}(y) \mm y \in K \mr$ is an open cover of $K$.
			By compactness of $K$, there is a finite $\ml y_1, \ldots, y_n\mr \subseteq K$
			such that $\bigcup_{i=1}^n B_{\eta_{y_i}}(y_i) \supseteq K$.
			We define
			\[ \eta := \min\ml \eta_{y_i} \mm i \in \ml 1, \ldots, n\mr \mr
			\quad\text{and}\quad
			K_\varepsilon := \bigcup_{i=1}^n B_{\eta_{y_i}}(y_i) \,.\]

			By construction, there is $i \in \ml 1, \ldots, n\mr$
			with $B_\eta(y) \subseteq B_{2\eta_{y_i}}(y)$
			for any $y \in K_\varepsilon$.
			The continuity of $\pi$ implies that there is $\delta > 0$ such that $\pi(B_\delta(x)) \subseteq B_{\eta}(\pi(x))$.
			Furthermore, the invariance of $d_Y$ implies
			$\beta(g,B_{\eta}(\pi(x))) = B_{\eta}(\beta(g,\pi(x)))$
			for all $g \in G$. 
			Let $H := \ml g \in G \mm \beta(g,\pi(x)) \in K_\varepsilon \mr$.
			For $g \in H$ there is $i \in \ml 1, \ldots, n \mr$ with 
			$\beta(g,\pi(B_\delta(x))) \subseteq B_\eta\kl \beta(g,\pi(x)) \kr \subseteq B_{2\eta_{y_i}}(y_i)$
			and thus
			\begin{align*}
				\alpha(g,B_\delta(x)) &\subseteq \pi\inv\kl \pi\kl \alpha(g,B_\delta(x)) \kr \kr \\ 
				&\subseteq \pi\inv\kl \beta\kl g, \pi(B_\delta(x)) \kr \kr \\
				&\subseteq \pi\inv\kl B_{2\eta_{y_i}}(y_i) \kr \\
				&\subseteq B_{\frac{\varepsilon}{2}}\kl \pi\inv\kl \ml y_i \mr \kr \kr
			\end{align*}
			Now let $u_1,\ldots,u_{m+1} \in \alpha(g,B_\delta(x))$.
			Then there is $x_1,\ldots,x_{m+1} \in \pi\inv \kl\ml y_i \mr\kr$ such that $d(x_j,u_j) < \frac{\varepsilon}2$.
			As $\pi\inv \kl\ml y_i \mr\kr$ contains only $m$ points, the pigeon hole principle yields the existence of $j \neq k$ such that $d(u_j,u_k) < \varepsilon$. 
			This implies now that $\diam_{m+1}\kl \alpha(g,B_\delta(x)) \kr < \varepsilon$ for any $g \in H$, i.e.
			\begin{align}
				H \subseteq \ml g \in G \mm \diam_{m+1}(\alpha(g,U)) \leqslant \varepsilon \mr \,.
				\label{eq:hittingtimescontained}
			\end{align}
			
			However, $H$ has a lower Banach density of at least $1-\varepsilon$ as the following simple application of the Uniform Ergodic \refthm{uniformergodicthm} shows.
			As $K_\varepsilon$ is open and $K$ is closed,
			Urysohn's \reflem{urysohn} shows that there is a continuous function $f:X\to[0,1]$ such that $f({K_\varepsilon^c}) = \ml 0 \mr$
			and $f(K) = \ml 1 \mr$.

			\refcor{uniformergodicandbanach} implies
			\begin{align*}
				\BDu_\Fc\kl H^c \kr < \varepsilon \,.
			\end{align*}
			Now \refe{eq:hittingtimescontained} implies that $x$ is Banach $\Fc$-diam-mean $(m+1)$-equicontinuous.
		\item[\refi{strongdiammean} implies \refi{weakdiammean}]
			This is immediate.

		\item[\refi{weakdiammean} implies \refi{almostsurelym1}]
			We proceed by contradiction.
			Fix any F\o lner sequence.
			So assume that $\pi$ is not almost surely $m$:1 and that $x \in X$ is weakly $\Fc$-diam-mean $(m+1)$-equicontinuous.
			Let $\varepsilon > 0$ be according to \refprop{notalmostsurelym1impliesnotbanachdiammean}.

			By \reflem{continuityofdiam}, $\diam_{m+1}$ is continuous, so
			\[ \Omega := \diam_{m+1}\inv\kl \il 0,\varepsilon \kr \kr = \ml K \in \Kc(X) \mm \diam_{m+1}(K) < \varepsilon \mr \]
			is open.
			Recall that the action of $\alpha$ on $X$ induces an action $\hat\alpha$ on $\Kc(X)$ by
			\[ \hat\alpha(g,K) = \ml \alpha(g,k) \mm k \in K \mr \,. \]
			By assumption on $x$, there is $\delta > 0$ and $(b_n)_{n\in\N} \in G^\N$ such that
			\begin{align}
				\label{eq:contradictionassumptiondiammean}
				\lim_{n\to\infty} \frac{\kl \ml g \in b_nF_n \mm \diam_{m+1}(\alpha(g,B_\delta(x))) < \varepsilon \mr \kr}{m(b_nF_n)} > 1 - \varepsilon \,.
			\end{align}
			Let $B := \overline{B_{\frac{\delta}{2}}(x)} \subseteq B_\delta(x)$.
			As in the proof of \refthm{mainthmfreqstable} \enquote{\refi{weakfrequentmstability} implies \refi{almostm1}} there is
			an ergodic measure $M^*_e$ and $B_* \in J$ such that $M^*_e( \ml K \in J \mm B_* \subseteq K \mr) > 0$ and
			\begin{align}
				\label{eq:contradictionstepdiammean}
				M^*_e(\Omega) > 1-\varepsilon \,.
			\end{align}
			Let $\Fc_{B_*}$ be the shifted F\o lner sequence given by \refprop{notalmostsurelym1impliesnotbanachdiammean}.
			The Lindenstrauss Ergodic \refthm{lindenstrauss} applied to a tempered subsequence of $\Fc_{B_*}$
			and $\1_{\Omega^c}$
			implies that
			\[ \frac{1}{m(F_{B_*,n})} \int_{F_{B_*,n}} \1_{\Omega^c}\kl \hat\alpha(g,K) \kr \dd m(g) \xlongrightarrow{n\to\infty} M_e^*(\Omega^c) \]
			For almost all $K \in \Kc(X)$.
			As $M^*_e( \ml K \in J \mm B_* \subseteq K \mr) > 0$, we can assume that $K \supseteq B_*$ and this implies
			\begin{align*}
				\label{eq:contradictiondiammean}
				M^*_e(\Omega^c) > 2\varepsilon \,.
			\end{align*}
			This yields a contradiction to \refe{eq:contradictionstepdiammean}.
			\qedhere
	\end{description}
\end{proof}

\bibliographystyle{plain}
%\nocite{*}
\bibliography{Literatur}
\end{document}